\documentclass[12pt]{article}

\usepackage{dsfont,amsmath,amsfonts,amsthm,amssymb,cite}
\usepackage[T1]{fontenc}
\usepackage[toc,page]{appendix} 
\usepackage{graphicx}
\usepackage{graphics}
\usepackage{stmaryrd}
\usepackage{verbatim}
\usepackage{hyperref}
\title{
         \huge{General Fragmentation Trees}}
\author{Robin Stephenson\footnote{CEREMADE, Université Paris-Dauphine}}
\begin{document}

\newtheorem{theo}{Theorem}[section]
\newtheorem{lemma}{Lemma}[section]
\newtheorem{prop}{Proposition}[section]
\newtheorem{cor}{Corollary}[section]
\newtheorem{defi}{Definition}[section]
\newtheorem{rem}{Remark}
\numberwithin{equation}{section}
\newcommand{\e}{{\mathrm e}}
\newcommand{\rep}{{\mathrm{rep}}}
\newcommand{\R}{{\mathbb{R}}}
\newcommand{\T}{\mathcal{T}}
\newcommand{\TT}{\mathbb{T}}
\newcommand{\s}{\mathcal{S}^{\downarrow}}
\newcommand{\p}{\mathcal{P}}
\newcommand{\N}{\mathbb{N}}
\newcommand{\D}{\mathcal{D}}
\newcommand{\TREE}{\mathrm{TREE}}
\maketitle

\begin{abstract}
We show that the genealogy of any self-similar fragmentation process can be encoded in a compact measured $\R$-tree. Under some Malthusian hypotheses, we compute the fractal Hausdorff dimension of this tree through the use of a natural measure on the set of its leaves. This generalizes previous work of Haas and Miermont which was restricted to conservative fragmentation processes.
\end{abstract}

\section{Introduction}
In this work, we study a family of trees derived from self-similar fragmentation processes. Such processes describe the evolution of an object which constantly breaks down into smaller fragments, each one then evolving independently from one another, just as the initial object would, but with a rescaling of time by the size of the fragment to a certain power called the index of self-similarity. This breaking down happens in two ways: erosion, a process by which part of the object is continuously being shaved off and thrown away, and actual splittings of fragments which are governed by a Poisson point process. Erosion is parametered by a nonnegative number $c$ called the erosion rate, while the splitting Poisson point process depends on a dislocation measure $\nu$ on the space \[\s=\{\mathbf{s}=(s_i)_{i\in\N}: s_1\geq s_2\geq\ldots\geq0,\sum s_i\leq 1\}.\] Precise definitions can be found in the main body of the article.

Our main inspiration is the 2004 article of Bénédicte Haas and Grégory Miermont \cite{HM04}. Their work focused on \emph{conservative} fragmentations, where there is no erosion and splittings of fragments do not change the total mass. They have shown that, when the index of self-similarity is negative, the genealogy of a conservative fragmentation process can be encoded in a continuum random tree, the genealogy tree of the fragmentation, which is compact and naturally equipped with a probability measure on the set of its leaves. Our main goal here will be to generalize the results they have obtained to the largest possible class of fragmentation processes: the conservation hypothesis will be discarded, though the index of self-similarity will be kept negative. We will show (Theorem \ref{11}) that we can still define some kind of fragmentation tree, but its natural measure will not be supported by the leaves, and we thus step out of the classical continuum random tree context set in \cite{A3}.

That the measure of a general fragmentation tree gives mass to its skeleton will be a major issue in this paper, and its study will therefore involve creating a new measure on the leaves of the tree. To do this we will restrict ourselves to \emph{Malthusian} fragmentations. Informally, for a fragmentation process to be Malthusian means that there is a number $p^*\in(0,1]$ such that, infinitesimally, calling $(X_i(t))_{i\in \N}$ the sizes of the fragments of the process at time $t$, the expectation of $\sum_{i\in\N} X_i(t)^{p^*}$ is constant. This allows us to use martingale methods and define a Malthusian measure $\mu^*$ on the leaves of the tree. The use of this measure then lets us obtain the fractal Hausdorff dimension of the set of leaves of the fragmentation tree, under a light regularity condition, which we will call "assumption $\mathbf{(H)}$":

\medskip
\begin{quote} The function $\psi$ defined on $\R$ by $\psi(p)= cp + \int_{\s} (1-\sum_i s_i^{p})d\nu(\mathbf{s}) \in[-\infty,+\infty)$ takes at least one finite strictly negative value on the interval $[0,1]$.\end{quote}
\begin{theo}\label{00} Assume $\mathbf{(H)}$. Then, almost surely, if the set of leaves of the fragmentation tree derived from an $\alpha$-self-similar fragmentation process with erosion rate $c$ and dislocation measure $\nu$ is not countable, its Hausdorff dimension is equal to $\frac{p^*}{|\alpha|}.$
\end{theo}
In \cite{HM04}, a dimension of $\frac{1}{|\alpha|}$ was found for conservative fragmentation trees, also under a regularity condition. We can see that non-conservation of mass makes the tree smaller in the sense of dimension. Note as well that the event where the leaves of the tree are countable only has positive probability if $\nu(0,0,\ldots,0)>0$, that is, if a fragment can suddenly disappear without giving any offspring.

\bigskip

\noindent\emph{Note:} in this paper, we use the convention that, when we take $0$ to a nonpositive power, the result is $0$. We therefore abuse notation slightly by omitting an indicator function such as $\mathbf{1}_{x\neq0}$ most of the time. In particular, sums such as $\sum_{i\in\N} x_i^p$ are implicitly taken on the set of $i$ such that $x_i\neq0$ even when $p\leq 0$.

\section{Background, preliminaries and some notation}
\subsection{Self-similar fragmentation processes}
\subsubsection{Partitions}

We are going to look at two different kinds of partitions. The first ones are \emph{mass partitions}. These are nonincreasing sequences $\mathbf{s}=(s_1,s_2,\ldots)$ with $s_i\geq0$ for every $i$ and such that $\sum_i s_i\leq 1$. These are to be considered as if a particle of mass $1$ had split up into smaller particles, some of its mass having turned into \emph{dust} which is represented by $s_0=1-\sum_i s_i$. We call $\s$ the set of mass partitions, it can be metrized with the restriction of the uniform norm and is then compact.

The more important partitions we will consider here are the set-theoretic partitions of finite and countable sets. For such a set $S$, we let $\p_S$ be the set of partitions of $S$. The main examples are of course the cases of partitions of $\N=\{1,2,3,\ldots\}$ (for countable sets) and, for $n\in\N$, $[n]=\{1,2,\ldots,n\}$. Let us focus here on $\p_{\N}$. A partition $\pi\in\p_{\N}$ will be written as a countable sequence of subsets of $\N$, called the blocks of the partition: $\pi=(\pi_1,\pi_2,\ldots)$ where every intersection between two different blocks is empty and the union of all the blocks is $\N$. The blocks are ordered by increasing smallest element: $\pi_1$ is the block containing $1$, $\pi_2$ is the block containing the smallest integer not in $\pi_1$, and so on. If $\pi$ has finitely many blocks, we complete the sequence with an infinite repeat of the empty set. (When not referring to a specific partition, the word "block" simply means "subset of $\N$".)

A partition can also be interpreted as an equivalence relation on $\N$: for a partition $\pi$ and two integers $i$ and $j$, we will write $i\sim_{\pi}j$ if $i$ and $j$ are in the same block of $\pi$.
We will also call $\pi_{(i)}$ the block of $\pi$ containing $i$.

We now have two ways to identify the blocks of a partition $\pi$: either with their rank in the partition's order or with their smallest element. Most of the time one will be more useful than the other, but sometimes we will want to mix both, which is why we will call $\rep(\pi)$ the set of smallest elements of blocks of $\pi$.

Let $B$ be a block. For all $\pi \in \p_{\N}$, we let $\pi\cap B$ be the restriction of $\pi$ to $B$, i.e. the partition of $B$ whose blocks are, up to reordering, the $(\pi_i \cap B)_{i\in\N}$.

We say that a partition $\pi$ is \emph{finer} than another partition $\psi$ if every block of $\pi$ is a subset of a block of $\psi$. This defines a partial order on the set of partitions.

Intersection and union operators can be defined on partitions: let $X$ be a set and, for $x\in X$, $\Pi_x$ be a partition. Then we define $\underset{x\in X}\cap \Pi_x$ to be the unique partition $\Psi$ such that, $\forall i,j\in\N, i\sim_{\Psi}j \Leftrightarrow \forall x\in X, i\sim_{\Pi_x}j$. The blocks of $\underset{x\in X}\cap \Pi_x$ are the intersections of blocks of the $(\Pi_x)_{x\in X}$. Similarly, assuming that all the $(\Pi_x)_{x\in X}$ are comparable, then we define $\underset{x\in X}\cup \Pi_x$ to be the unique partition $\Psi$ such that, $\forall i,j\in\N, i\sim_{\Psi}j \Leftrightarrow \exists x\in X, i\sim_{\Pi_x}j$.

We endow $\p_{\N}$ with a metric: for two partitions $\pi$ and $\pi'$, let $n(\pi,\pi')$ be the highest integer $n$ such that $\pi\cap[n]$ and $\pi'\cap[n]$ are equal ($n(\pi,\pi')=\infty$ if $\pi=\pi'$) and let $d(\pi,\pi')=2^{-n(\pi,\pi')}$. This defines a distance function on $\p_{\N}$, which in fact satisfies the ultra-metric triangle inequality. This metric provides a topology on $\p_{\N}$, for which convergence is simply characterized: a sequence $(\pi_n)_{n\in\N}$ of partitions converges to a partition $\pi$ if, and only if, for every $k$, there exists $n_k$ such that $\pi_{n}\cap[k]=\pi\cap [k]$ for $n$ larger than $n_k$. The metric also provides $\p_{\N}$ with a Borel $\sigma$-field, which is easily checked to be the $\sigma$-field generated by the restriction maps, i.e. the functions which which map $\pi$ to $\pi\cap[n]$ for all integers $n$.

Let $S$ and $S'$ be two sets with a bijection $f: S \to S'$. Then we can easily transform partitions of $S'$ into partitions of $S$: let $\pi$ be a partition of $S'$, we let $f\pi$ be the partition defined by: $\forall i,j\in S, i\sim_{f\pi}j \Leftrightarrow f(i)\sim_{\pi}f(j)$. This can be used to generalize the metric $d$ to $\p_S$ for infinite $S$ (note that the notion of convergence does not depend on the chosen bijection), and then $\pi \mapsto f\pi$ is easily seen to be continuous.

Special attention is given to the case where $f$ is a permutation: we call \emph{permutation} of $\N$ any bijection $\sigma$ of $\N$ onto itself. A $\p_{\N}$-valued random variable (or random partition) $\Pi$ is said to be \emph{exchangeable} if, for all permutations $\sigma$, $\sigma\Pi$ has the same law as $\Pi$.

Let $B$ be a block. If the limit $\lim_{n\to\infty} \frac{1}{n}\#(B\cap[n])$ exists then we write it $|B|$ and call it the \emph{asymptotic frequency} or more simply \emph{mass} of $B$. If all the blocks of a partition $\pi$ have asymptotic frequencies, then we call $|\pi|^{\downarrow}$ their sequence in decreasing order, which is an element of $\s$. This defines a measurable, but not continuous, map.

A well-known theorem of Kingman \cite{Kingman} links exchangeable random partitions of $\N$ and random mass partitions through the "paintbox construction". More precisely: let $\mathbf{s}\in\s$, and $(U_i)_{i\in\N}$ be independent uniform variables on $[0,1]$, we define a random partition $\Pi_{\mathbf{s}}$ by $\forall i\neq j,\, i\sim_{\Pi_\mathbf{s}}j \Leftrightarrow \exists k,\; U_i,U_j\in[\sum_{p=1}^{k}s_p,\sum_{p=1}^{k+1}s_p[$. This random partition is exchangeable, all its blocks have asymptotic frequencies, and $|\Pi_{\mathbf{s}}|^{\downarrow}=\mathbf{s}$. By calling $\kappa_{\mathbf{s}}$ the law of $\Pi_{\mathbf{s}}$, Kingman's theorem states that, for any exchangeable random partition $\Pi$, there exists a random mass partition $S$ such that, conditionally on $S$, $\Pi$ has law $\kappa_S$. A useful consequence of this theorem is found in \cite{Ber}, Corollary $2.4$: for any integer $k$, conditionally on the variable $S$, the asymptotic frequency $|\Pi_{(k)}|$ of the block containing $k$ exists almost surely and is a size-biased pick amongst the terms of $S$, which means that its distribution is $\sum_i S_i\delta_{S_i} + S_0\delta_{S_0}$ (with $S_0=1-\underset{i\in\N}\sum S_i$).

Let $\Pi$ and $\Psi$ be two independent exchangeable random partitions. Then, for any $i$ and $j$, the block $\Pi_i\cap\Psi_j$ of $\Pi\cap\Psi$ almost surely has asymptotic frequency $|\Pi_i||\Psi_j|$. This stays true if we take countably many partitions, as is stated in \cite{Ber}, Corollary 2.5.

\subsubsection{Definition of fragmentation processes}
Partition-valued fragmentation processes were first introduced in \cite{B01} (homogeneous processes only) and \cite{B02} (the general self-similar kind).
\begin{defi} Let $(\Pi(t))_{t\geq0}$ be a $\p_{\N}$-valued process with càdlàg paths, which satisfies $\Pi(0)=((\N,\emptyset, \emptyset,\ldots))$, which is exchangeable as a process (i.e. for all permutations $\sigma$, the process $(\sigma\Pi(t))_{t\geq0}$ has the same law as $(\Pi(t))_{t\geq0}$ ) and such that, almost surely, for all $t\geq0$, all the blocks of $\Pi(t)$ have asymptotic frequencies. Let $\alpha$ be any real number. We say that $\Pi$ is a \emph{self-similar fragmentation process with index $\alpha$} if it also satisfies the following self-similar fragmentation property: for all $t\geq0$, given $\Pi(t)=\pi$, the processes $\big(\Pi(t+s)\cap \pi_i\big)_{s\geq0}$ (for all integers $i$) are mutually independent, and each one has the same law as $\big(\Pi(|\pi_i|^{\alpha}(s)) \cap \pi_i\big)_{s\geq0}$.
\end{defi}

When $\alpha=0$, we will say that $\Pi$ is a \emph{homogeneous} fragmentation process instead of $0$-self-similar fragmentation process.

\begin{rem} One can give a Markov process structure to an $\alpha$-self-similar fragmentation process $\Pi$ by defining, for any partition $\pi$, the law of $\Pi$ starting from $\pi$. Let $(\Pi^i)_{i\in\N}$ be independent copies of $\Pi$ (each one starting at $(\N,\emptyset,\ldots)$ ), then we let, for all $t\geq0$, $\Pi(t)$ be the partition whose blocks are exactly those of $((\Pi^i(|\pi_i|^{\alpha}t)\cap\pi_i)_{i\in\N}$. In this case the process isn't exchangeable with respect to all permutations of $\N$, but only with respect to permutations which stabilize the blocks of the initial value $\pi$.
\end{rem}

Fragmentation processes are seen as random variables in the Skorokhod space $\mathcal{D}=\mathcal{D}([0,+\infty),\p_{\N})$, which is the set of càdlàg functions from $[0,+\infty)$ to $\p_{\N}$. An element of $\D$ will typically be written as $(\pi_t)_{t\geq0}$. This space can be metrized with the \emph{Skorokhod metric} and is then Polish. More importantly, the Borel $\sigma$-algebra on $\D$ is then the $\sigma$-algebra spanned by the \emph{evaluation functions} $(\pi_t)_{t\geq0}\mapsto\pi_s$ (for $s\geq0$), implying that the law of a process is characterized by its finite-dimensional marginal distributions. The definition of the Skorokhod metric and generalities on the subject can be read in \cite{JSh}, Section VI.1.

Let us give a lemma which makes self-similarity easier to handle at times:

\begin{lemma} Let $(\Pi(t))_{t\geq0}$ be any exchangeable $\p_{\N}$-valued process, and $A$ any infinite block. Take any bijection $f$ from $A$ to $\N$, then the two $\p_{A}$-valued processes $(\Pi(t)\cap A)_{t\geq0}$ and $(f\Pi(t))_{t\geq0}$ have the same law.
\end{lemma}
\begin{proof} For all $n\in\N$, let $A_n=\{f^{-1}(1),f^{-1}(2),\ldots,f^{-1}(n)\}$. Recall then that, with the $\sigma$-algebra which we have on $\p_A$, we only need to check that, for all $n\in\N$, $(\Pi_{|A_n})$ has the same law as $f(\Pi\cap[n])$. If $G$ is a nonnegative measurable function on $\mathcal{D}([0,+\infty),\p_{A_n})$, we have, by using the fact that the restriction of $f$ from $[n]$ to $A_n$ can be extended to a bijection of $\N$ onto itself
	\[E[G(\Pi\cap A_n)]=E\big[G((f\Pi)\cap A_n)\big] =E\big[G(f(\Pi\cap[n]))\big],
\]
which is all we need.
\end{proof}

This lemma will make it easier to show the fragmentation property for some $\D$-valued processes we will build throughout the article.




\subsubsection{Characterization and Poissonian construction}

A famous result of Bertoin (detailed in \cite{Ber}, Chapter 3) states that the law of a self-similar fragmentation process is characterized by three parameters: the index of self-similarity $\alpha$, an \emph{erosion coefficient} $c\geq0$ and a \emph{dislocation measure} $\nu$, which is a $\sigma$-finite measure on $\s$ such that
	\[ \nu(1,0,0,\ldots)=0\text{  and  } \int_{\s}(1-s_1)d\nu(\mathbf{s})<\infty.
\]

Bertoin's result can be formulated this way: for any fragmentation process, there exists a unique triple $(\alpha,c,\nu)$ such that our process has the same law as the process which we are about to explicitly construct.

First let us describe how to build a fragmentation process with parameters $(0,0,\nu)$ which we will call $\Pi^{0,0}$. Let $\kappa_{\nu}(d\pi)=\int_{\s} \rho_\mathbf{s}(d\pi)d\nu(\mathbf{s})$ where $\kappa_s(d\pi)$ denotes the paintbox measure on $\mathcal{P}_{\mathbb{N}}$ corresponding to $\mathbf{s} \in \s$. For every integer $k$, let $(\Delta^k_t)_{t\geq0}$ be a Poisson point process with intensity $\kappa_\nu$, such that these processes are all independent. Now let $\Pi^{0,0}(t)$ be the process defined by $\Pi^{0,0}(0)=(\mathbb{N},\emptyset,\emptyset, \ldots)$ and which jumps when there is an atom $(\Delta^k_t)$: we replace the $k$-th block of $\Pi^{0,0}(t-)$ by its intersection with $\Delta^k_t$. This might not seem well-defined since the Poisson point process can have infinitely many atoms. However, one can check (as we will do in Section $5.2$ in a slightly different case) that this is well defined by restricting to the first $N$ integers and taking the limit when $N$ goes to infinity.

To get a $(0,c,\nu)$-fragmentation which we will call $\Pi^{0,c}$, take a sequence $(T_i)_{i \in \mathbb{N}}$ of exponential variables with parameter $c$ which are independent from each other and independent from $\Pi^{0,0}$. Then, for all $t$, let $\Pi^{0,c}(t)$ be the same partition as $\Pi^{0,0}(t)$ except that we force all integers $i$ such that $T_i\leq t$ to be in a singleton if they were not already.

Finally, an $(\alpha,c,\nu)$-fragmentation can then be obtained by applying a Lamperti-type time-change to all the blocks of $\Pi^{0,c}$: let, for all $i$ and $t$,
	\[\tau_i(t) = \inf \Big\{u, \int_0^u |\Pi_{(i)}^{0,c}(r)|^{-\alpha}dr >t\Big\}.
\] 
Then, for all $t$, let $\Pi^{\alpha,c}(t)$ be the partition such that two integers $i$ and $j$ are in the same block of $\Pi^{\alpha,c}(t)$ if and only if $j\in\Pi_{(i)}^{0,c}(\tau_i(t))$. Note that if $t\geq \int_0^{\infty} |\Pi_{(i)}^{0,c}(r)|^{-\alpha}dr$, then the value of $\tau_i(t)$ is infinite, and $i$ is in a singleton of $\Pi^{\alpha,c}(t)$. Note also that the time transformation is easily invertible: for $s\in [0,\infty)$, we have
	\[\tau_i^{-1}(s) = \inf \Big\{u, \int_0^u |\Pi_{(i)}^{\alpha,c}(r)|^{+\alpha}dr >s\Big\}.
\] 
\medskip

This time-change can in fact be done for any element $\pi$ of $\D$: since, for all $i\in\N$ and $t\geq0$, $\tau_i(t)$ is a measurable function of $\Pi^{0,c}$, there exists a measurable function $G^{\alpha}$ from $\D$ to $\D$ which maps $\Pi^{0,c}$ to $\Pi^{\alpha,c}$.

Let us once and for all fix our notations for the processes: in this article, $c$ and $\nu$ will be fixed (with $c\neq0$ or $\nu\neq0$ to remove the trivial case), however we will often jump between a homogeneous $(0,c,\nu)$-fragmentation and the associated self-similar $(\alpha,c,\nu)$-fragmentation. This is why we will rename things and let $\Pi=\Pi^{0,c}$ as well as $\Pi^{\alpha}=\Pi^{\alpha,c}$. We then let $(\mathcal{F}_t)_{t\geq0}$ be the canonical filtration associated to $\Pi$ and $(\mathcal{G}_t)_{t\geq0}$ the one associated to $\Pi^{\alpha}$.

\subsubsection{A few key results}
One simple but important consequence of the Poissonian construction is that the notation $|\Pi^{\alpha}_{(i)}(t^-)|$ is well-defined for all $i$ and $t$: it is equal to both the limit, as $s$ increases to $t$, of $|\Pi^{\alpha}_{(i)}(s)|$, and the asymptotic frequency of the block of $\Pi^{\alpha}(t^-)$ containing $i$.

For every integer $i$, let $\mathcal{G}_i$ be the canonical filtration of the process $(\Pi^{\alpha}_{(i)}(t))_{t\geq0}$, and consider a family of random times $(L_i)_{i\in\N}$ such that $L_i$ is a $\mathcal{G}_i$-stopping time for all $i$. We say that $(L_i)_{i\in\N}$ is a \emph{stopping line} if, for all integers $i$ and $j$, $j\in \Pi^{\alpha}_{(i)}(L_i)$ implies $L_i=L_j$. Under this condition, $\Pi^{\alpha}$ then satisfies an extended fragmentation property (proved in \cite{Ber}, Lemma 3.14): we can define for every $t$ a partition $\Pi^{\alpha}(L+t)$ whose blocks are the $(\Pi^{\alpha}_{(i)}(L_i+t))_{i\in\N}$. Then conditionally on the sigma-field $\mathcal{G}_L$ generated by the $\mathcal{G}_i(L_i)$ ($i\in\N$), the process $(\Pi^{\alpha}(L+t))_{t\geq0}$ has the same law as $\Pi$ started from $\Pi^{\alpha}(L)$.

\smallskip
One of the main tools of the study of fragmentation processes is the \emph{tagged fragment}: we specifically look at the block of $\Pi^{\alpha}$ containing the integer $1$ (or any other fixed integer). Of particular interest, its mass can be written in terms of L\'evy processes: one can write, for all $t$, $|\Pi^{\alpha}_{(1)}(t)|=\e^{-\xi_{\tau(t)}}$ where $\xi$ is a killed subordinator with Laplace exponent $\phi$ defined for nonnegative $q$ by
	\[\phi(q)= c(q+1) + \int_{\s}(1-\sum_{n=1}^{\infty}s_n^{q+1})d\nu(\mathbf{s}),
\] and $\tau(t)$ is defined for all $t$ by $\tau(t)=\inf \Big\{u, \int_0^u \e^{\alpha\xi_r}dr >t\Big\}.$ Note that standard results on Poisson measures then imply that, if $q\in\R$ is such that $\int_{\s}(1-\sum_{n=1}^{\infty}s_n^{q+1})d\nu(\mathbf{s})>-\infty$, then we still have $E[e^{-q\xi_t}\mathbf{1}_{\{\xi_t<\infty\}}]=e^{-t\phi(q)}.$

In particular, the first time $t$ such that the singleton $\{1\}$ is a block of $\Pi^{\alpha}(t)$ is equal to $\int_0^{\infty}\e^{\alpha\xi_s}ds$, the exponential functional of the Lévy process $\alpha\xi$, which has been studied for example in \cite{CPY97}. In particular it is finite a.s. whenever $\alpha$ is strictly negative and $\Pi$ is not constant.

\subsection{Random trees}
\subsubsection{$\R$-trees}
\begin{defi} Let $(\T,d)$ be a metric space. We say that it is an \emph{$\R$-tree} if it satisfies the following two conditions:

\textbullet \hspace{2 mm} for all $x,y \in \T$, there exists a unique distance-preserving map $\phi_{x,y}$ from $[0,d(x,y)]$ into $\mathcal{T}$ such  $\phi_{x,y}(0)=x$ and $\phi_{x,y}(d(x,y))=y;$

\textbullet \hspace{2 mm} for all continuous and one-to-one functions $c$: $[0,1] \to \T $, we have $\\ c([0,1]) = \phi_{x,y}([0,d(x,y)]),$ where $x=c(0)$ and $y=c(1)$.

\end{defi}
For any $x,y$ in a tree, we will denote by $\llbracket x,y\rrbracket$ the image of $\phi_{x,y}$, i.e. the path between $x$ and $y$.
Here is a simple characterization of $\R$-trees which we will use in the future. It can be found in \cite{Ev}, Theorem 3.40.

\begin{prop} A metric space $(\T,d)$ is an $\R$-tree if and only if it is connected and satisfies the following property, called the \emph{four-point condition}:
	\[\forall x,y,u,v \in \T, d(x,y)+d(u,v)\leq \max\big(d(x,u)+d(y,v),d(x,v)+d(y,u)\big).
\]
\end{prop}

By permuting $x,y,z,t$, one gets a more explicit form of the four-point condition: out of the three numbers $d(x,y)+d(u,v)$, $d(x,u)+d(y,v)$ and $d(x,v)+d(y,u)$, at least two are equal, and the third one is smaller than or equal to the other two.

For commodity we will, for an $\R$-tree $(\T,d)$ and $a>0$, call $a\T$ the $\R$-tree $(\T,ad)$ which is the same tree as $\T$, except that all distances have been rescaled by $a$.

\subsubsection{Roots, partial orders and height functions}
All the trees which we will consider will be \emph{rooted}: we will fix a distinguished vertex $\rho$ called the \emph{root}. This provides $\T$ with a \emph{height function} $ht$ defined by $ht(x)=d(\rho,x)$ for $x\in\T$.

We use the height function to define, for $t\geq0$, the subset $\T_{\leq t}=\{x\in \T: \; ht(x)\leq t\}$, as well as the similarly defined $\T_{<t}$, $\T_{\geq t}$ and $\T_{>t}$. Note that $\T_{\leq t}$ and $\T_{<t}$ are both $\R$-trees, as well as the connected components of $\T_{\geq t}$ and $\T_{>t}$, which we will call the \emph{tree components} of $\T_{\geq t}$ and $\T_{>t}$.

Having a root on $\T$ also lets us define a partial order, by declaring that $x\leq y$ if $x\in \llbracket \rho,y\rrbracket$. We will often say that $x$ is an ancestor of $y$ in this case, or simply that $x$ is lower than $y$. We can then define for any $x$ in $\T$ the \emph{subtree of $\T$ rooted at $x$}, which we will call $\T_x$: it is the set $\{y \in\T: y\geq x\}.$ We will also say that two points $x$ and $y$ are \emph{on the same branch} if they are comparable, i.e. if we have $x\leq y$ or $y \leq x$. For every subset $S$ of $\T$ we can define the \emph{greatest common ancestor} of $S$, which is the highest point which is lower than all the elements of $S$. The greatest common ancestor of two points $x$ and $y$ of $\T$ will be written $x\wedge y$.

One convenient property is that we can recover the metric from the order and the height function. Indeed, for any two points $x$ and $y$, we have $d(x,y)=ht(x)+ht(y)-2ht(x\wedge y)$.

We also call \emph{leaf} of $\T$ any point $L$ such that the $\T_L=\{L\}$. The set of leaves of $\T$ will be written $\mathcal{L}(\T)$, and its complement is called the \emph{skeleton} of $\T$.

\subsubsection{Gromov-Hausdorff distances, spaces of trees}

\smallskip

Recall that, if $A$ and $B$ are two compact nonempty subsets of a metric space $(E,d)$, then we can define the Hausdorff distance between $A$ and $B$ by 
	\[d_{E,H}(A,B)=\inf \{\epsilon>0; A\subset B_\epsilon\text{ and }B\subset A_\epsilon\},
\]
 where $A_{\epsilon}$ and $B_{\epsilon}$ are the closed $\epsilon$-enlargements of $A$ and $B$ (that is, $A_{\epsilon}=\{x\in E, \exists a\in A, d(x,a)\leq\epsilon\}$ and the corresponding definition for $B$).

Now, if one considers two compact rooted $\R$-trees $(\T,\rho,d)$ and $(\T',\rho',d')$, define their \emph{Gromov-Hausdorff distance}:
	\[d_{GH}(\T,\T') = \inf [ \max (d_{\mathcal{Z},H} (\phi(\T),\phi'(\T')), d_\mathcal{Z}(\phi(\rho),\phi'(\rho')))],
\]
where the infimum is taken over all pairs of isometric embeddings $\phi$ and $\phi'$ of $\T$ and $\T'$ in the same metric space $(\mathcal{Z},d_{\mathcal{Z}}).$

We will also want to consider pairs $(\T,\mu)$, where $\T$ ($d$ and $\rho$ being implicit) is a compact rooted $\R$-tree and $\mu$ a Borel probability measure on $\T$. Between two such compact rooted measured trees $(\T,\mu)$ and $(\T',\mu')$, one can define the \emph{Gromov-Hausdorff-Prokhorov} distance by

  \[d_{GHP}(\T,\T') = \inf [ \max (d_{\mathcal{Z},H} (\phi(\T),\phi'(\T')), d_\mathcal{Z}(\phi(\rho),\phi'(\rho')),d_{\mathcal{Z},P}(\phi_*\mu,\phi'_*\mu')],
\]
where the infimum is taken on the same space, and $d_{\mathcal{Z},P}$ denotes the Prokhorov distance between two Borel probability measures on $\mathcal{Z}$. The only thing we need about this metric is that convergence for $d_{\mathcal{Z},P}$ is equivalence to convergence to weak convergence of Borel probability measures on $\mathcal{Z}$, see \cite{Bill}.

These two metrics allow for study of spaces of trees, and it can be shown (see \cite{EPW06} and \cite{ADH}) that these spaces are well-behaved.
\begin{prop} Let $\TT$ and $\TT_W$ be respectively the set of equivalence classes of compact rooted trees and the set of classes of compact rooted measured trees, where two trees are said to be equivalent if there is a root-preserving (and measure-preserving in the measured case) isometric bijection between them. Then $(\TT,d_{GH})$ and $(\TT_W,d_{GHP})$ are Polish spaces.
\end{prop}

\subsubsection{Decreasing functions and measures on trees}
Let us give a tool which will allow us to define measures on a compact rooted tree $\T$ only through their values on all the subtrees $\T_x$ for $x\in\T$. Let $m$ be a decreasing function on $\T$ taking values in $[0,\infty)$. One can easily define the left-limit $m(x^-)$ of $m$ at any point $x\in \T$, since $\llbracket \rho,x\rrbracket$ is isometric to a line segment, for example by setting $m(x^-)=\underset{t\to ht(x)^-}\lim m(\phi_{\rho,x}(t))$. Let us also define the \emph{additive right-limit} $m(x^+)$: since $\T$ is compact, the set $\T_x\setminus\{x\}$ has countably many connected components, say $(\T_i)_{i\in S}$ for a finite or countable set $S$. Let, for all $i\in S$, $x_i\in\T_i$. We then set 
	\[m(x^+)= \underset{i\in S}\sum \lim_{t\to ht(x)^+} m(\phi_{\rho,x_i}(t)).
\]
 This is well-defined, because it does not depend on our choice of $x_i\in\T_i$ for all $i$. We say that $m$ is \emph{left-continuous} at a point $x$ if $m(x^-)=m(x)$.

\begin{prop}\label{01} Let $m$ be a decreasing, positive and left continuous function on $\T$ such that, for all $x\in\T$, $m(x)\geq m(x^+)$. Then there exists a unique Borel measure $\mu$ on $\T$ such that
	\[\forall x\in \T, \mu(\T_x)=m(x).
\]
\end{prop}

While the idea behind the proof of Proposition \ref{01} is fairly simple, the proof itself is relatively involved and technical, which is why we postpone it for Appendix A.

\section{The fragmentation tree}
\subsection{Main result}

We are going to show a bijective correspondence between the laws of fragmentation processes with negative index and a certain class of random trees. We fix from now on an index $\alpha<0.$ If $(\T,\mu)$ is a measured tree and $S$ is a measurable subset of $\T$ with $\mu(S)>0$, we let $\mu_S$ be the measure $\mu$ conditioned on $S$.
\begin{defi}\label{def1} Let $(\T,\mu)$ be a random variable in $\TT_W$. For all $t\geq0$, let $\T_1(t),\T_2(t),\ldots$ be the connected components of $\T_{>t}$, and let, for all $i$, $x_i(t)$ be the point of $\T$ with height $t$ which makes $\T_i(t)\cup\{x_i(t)\}$ connected. We say that $\T$ is \emph{self-similar with index $\alpha$} if $\mu(\T_i(t))>0$ for all choices of $t\geq0$ and $i$ and if, for any $t\geq0$, conditionally on $\big(\mu(\T_i(s))\big)_{i\in\N,s\leq t}$, the trees $(\T_i(t)\cup\{x_i(t)\},\mu_{\T_i(t)})_{i\in\N}$ are independent and, for any $i$, $(\T_i(t)\cup \{x_i(t)\},\mu_{\T_i(t)})$ has the same law as $(\mu(\T_i(t))^{-\alpha} \T',\mu')$ where $(\T',\mu')$ is an independent copy of $(\T,\mu)$.
\end{defi}
The similarity with the definition of an $\alpha$-self-similar fragmentation process must be pointed out: in both definitions, the main point is that each "component" of the process after a certain time is independent of all the others and has the same law as the initial process, up to rescaling. In fact, the following is an straightforward consequence of our definitions:

\begin{prop}\label{1more} Let $(\T,\mu)$ be a self-similar tree with index of similarity $\alpha$. Let $(P_i)_{i\in\N}$ be an exchangeable sequence of variables directed by $\mu$. Define for every $t\geq0$ a partition $\Pi_{\T}(t)$ by saying that $i$ and $j$ are in the same block of $\Pi_{\T}(t)$ if and only if $P_i$ and $P_j$ are in the same connected component of $\{x\in \T, ht(x)>t\}$ (in particular an integer $i$ is in a singleton if $ht(P_i)\leq t$). Then $\Pi_{\T}$ is an $\alpha$-self-similar fragmentation process.
\end{prop}

\begin{proof} First of all, we need to check that, for all $t\geq0$, $\Pi_{\T}(t)$ is a random variable. We therefore fix $t>0$ and notice that the definition of $\Pi_{\T}(t)$ entails that, for all $i\in\N$ and $j\in\N$, 
	\[i\sim_{\Pi_{\T}(t)} j \Leftrightarrow ht (P_i\wedge P_j)>t,
\]
which is a measurable event. Thus, for all integers $n$ and all partitions $\psi$ of $[n]$, the event $\{\Pi_{\T}(t)\cap[n]=\psi\}$ is also measurable. It then follows that $\Pi_{\T}(t)\cap[n]$ is measurable for all $n\in\N$, and therefore $\Pi_{\T}(t)$ itself is measurable.

Next we need to check that $\Pi_{\T}$ is càdlàg. It is immediate from the definition that $\Pi_{\T}$ is decreasing (in the sense that $\Pi_{\T}(s)$ is finer than $\Pi_{T}(t)$ for $s>t$), and then that, for any $t$, $\Pi_{\T}(t)=\underset{s>t}\cup \Pi_{\T}(s)$, and thus the process is right-continuous. Similarly, the process has a left-limit at $t$ for all $t$, which is indentified as $\Pi_{\T}(t^-)=\underset{s<t}\cap \Pi_{\T}(s)$.

Exchangeability as a process of $\Pi_{\T}$ is an immediate consequence of the exchangeability of the sequence $(P_i)_{i\in\N}$.

The fact that, almost surely, all the blocks of $\Pi_{\T}(t)$ for $t\geq0$ have asymptotic frequencies is a consequence of the Glivenko-Cantelli theorem (see \cite{Dud}, Theorem 11.4.2). For $i\geq 2$, let $Y_i=ht(P_1\wedge P_i)$, then, for $t< Y_i$, $1$ and $i$ are in the same block of $\Pi_{\T}(t)$, and for $t\geq Y_i$, they are not. Then we have, for all $t\geq0$,
	\[\#(\Pi_{\T}(t)\cap[n])_{(1)}= 1 + \sum_{i=2}^n \mathbf{1}_{Y_i>t}.
\]
It then follows from the Glivenko-Cantelli theorem (applied conditionally on $\T$, $\mu$ and $P_1$) that, with probability one, for all $t\geq0$, $\frac{1}{n}\#(\Pi_{\T}(t)\cap[n])_{(1)}$ converges as $n$ goes to infinity, the limit being the $\mu$-mass of the tree component of $\T_{>t}$ containing $P_1$ (or $0$ if $ht(P_1)<t$). By replacing $1$ with any integer $i$, we get the almost sure existence of the asymptotic frequencies of $\Pi_{\T}$ at all times.

Let us now check that $\Pi_{\T}(0)=(\N,\emptyset,\ldots)$ almost surely, which amounts to saying that $\T\setminus\{\rho\}$ is connected. Apply the self-similar fragmentation property at time $0$: the tree $\T_1(0)\cup\{\rho\}$ (as in Definition \ref{def1}) has the same law as $\T$ up to a random multiplicative constant, and $\T_1$ is almost surely connected by definition. Thus $\T\setminus\{\rho\}$ is almost surely connected. A similar argument also shows that $\mu(\{\rho\})$ is almost surely equal to zero.

Finally, we need to check the $\alpha$-self-similar fragmentation property for $\Pi_{\T}$. Let $t\geq0$ and $\pi=\Pi_{\T}(t)$. For every integer $k$, we let $i(k)$ be the unique integer such that $k\in\pi_{i(k)}$ and, for every $i$, we let $\T_i(t)$ be the tree component of $\T_{>t}$ containing the points $P_k$ with $k\in\N$ such that $i(k)=i$ (if $\pi_i$ is a singleton, then $\T_i(t)$ is the empty set). We also add the natural rooting point $x_i$ of $\T_i$. Since, for all $k$, $i(k)$ is measurable knowing $\Pi_{\T}(t)$, we get that, conditionally on $(\T,\mu)$ and $\Pi_{\T}(t)$, $P_k$ is distributed according to $\mu_{\T_{i(k)}}$. From the independence property in Definition \ref{def1} then follows that the $(\Pi_{\T}(t+.)\cap\pi_i)_{i\in\mathbb{N}}$ are independent. We now just need to identify their law. If $i\in\N$ is such that $\pi_i$ is a singleton then there is nothing to do. Otherwise $\pi_i$ is infinite: let $f$ be any bijection $\N\to\pi_i$, and rename the points $P_k$ with $k$ such that $i(k)=i$ by letting $P'_k=P_{f^(k)}$. By the self-similarity of the tree, the partition-valued process built from $\T_i\cup\{x_i\}$ and the $P'_j$ (with $j\in\N$) has the same law as $\Pi_{\T}(|\pi_i|^{-\alpha}s)_{s\geq0}$, and therefore $\Pi_{\T}(t+.)\cap \pi_i$ has the same law as $\big(f\Pi^i(|\pi_i|^{\alpha}s)\big)_{s\geq0}$, which is what we wanted.

\end{proof}

Our main result is a kind of converse of this proposition, in law.

\begin{theo}\label{11} Let $\Pi^{\alpha}$ be a non-constant fragmentation process with index of similarity $\alpha<0$. Then there exists a random $\alpha$-self-similar tree $(\T_{\Pi^{\alpha}},\mu_{\Pi^{\alpha}})$ such that $\Pi_{\T_{\Pi^{\alpha}}}$ has the same law as $\Pi^{\alpha}$.
\end{theo}

\begin{rem} This is analogous to a recent result obtained by Chris Haulk and Jim Pitman in \cite{HP11}, which concerns exchangeable hierarchies. An exchangeable hierarchy can be seen as a fragmentation of $\mathbb{N}$ where one has forgotten time. Haulk and Pitman show that, just as with self-similar fragmentations, in law, every exchangeable hierarchy can be sampled from a random measured tree.
\end{rem}

The rest of this section is dedicated to the proof of Theorem \ref{11}. We fix from now on a fragmentation process $\Pi^{\alpha}$ (defined on a certain probability space $\Omega$) and will build the tree $\T$ and the measure $\mu$ (now omitting the index $\Pi^{\alpha}$).

\subsection{The genealogy tree of a fragmentation}
We are here going to give an explicit description of $\T$ which has the caveat of not showing that $\T$ is a random variable, i.e. a $d_{GH}$-measurable function of $\Pi^{\alpha}$ (something we will do in the following section). Since this construction is completely deterministic, we will slightly change our assumptions and at first consider a single element $\pi$ of $\D$ which is decreasing (the partitions get finer with time). For every integer $i$, let $D_i$ be the smallest time at which $i$ is in a singleton of $\pi$ and for every block $B$ with at least two elements, let $D_B$ be the smallest time at which all the elements of $B$ are not in the same block of $\pi$ anymore. We will assume that $\pi$ is such that all these are finite.

\begin{prop} There is, up to bijective isometries which preserve roots, a unique complete rooted $\R$-tree $\T$ equipped with points $(Q_i)_{i\in\N}$ such that:

(i) For all $i$, $ht(Q_i)=D_i$.

(ii) For all pairs of integers $i$ and $j$, we have $ht(Q_i \wedge Q_j)= D_{\{i,j\}}$.

(iii) The set $\underset{i\in\N}\cup\llbracket\rho,Q_i\rrbracket$ is dense in $\T$.
\end{prop}

$\T$ will then be called the \emph{genealogy tree} of $\pi$ and for all $i$, $Q_i$ will be called the \emph{death point} of $i$.

\begin{proof} Let first prove the uniqueness of $\T$. We give ourselves another tree $\T'$ with root $\rho'$ and points $(Q'_i)_{i\in\N}$ which also satisfy $(i)$, $(ii)$ and $(iii)$. First note that, if $i$ and $j$ are two integers such that $Q_i=Q_j$, then $D_{\{i,j\}}=D_i=D_j$ and thus $Q'_i=Q'_j$. This allows us to define a bijection $f$ between the two sets $\{\rho\}\cup\{Q_i,\; i\in\N\}$ and $\{\rho'\}\cup\{Q'_i,\; i\in\N\}$ by letting $f(\rho)=\rho'$ and, for all $i$, $f(Q_i)=Q'_i$. Now recall that we can recover the metric from the height function and the partial order: we have, for all $i$ and $j$, $d(Q_i,Q_j)=D_i+D_j-2D_{\{i,j\}},$ and the same is true in $\T'$. Thus $f$ is isometric and we can (uniquely) extend it to a bijective isometry between $\underset{i\in\N}\cup\llbracket\rho,Q_i\rrbracket$ and $\underset{i\in\N}\cup\llbracket\rho',Q'_i\rrbracket$, by letting, for $i\in\N$ and $t\in [0,D_i]$, $f(\phi_{\rho,Q_i}(t))=\phi_{\rho',Q'_i}(t)$. To check that this is well defined, we just need to note that, if $i$, $j$ and $t$ are such that $\phi_{\rho,Q_i}(t)=\phi_{\rho,Q_j}(t)$, then $t\leq D_{\{i,j\}}$ and thus we also have $\phi_{\rho',Q'_i}(t)=\phi_{\rho',Q'_j}(t)$. This extension is still an isometry because it preserves the height and the partial order and is surjective by definition, thus it is a bijection. By standard properties of metric completions, $f$ then extends into a bijective isometry between $\T$ and $\T'$.

\medskip

To prove the existence of $\T$, we are going to give an abstract construction of it. Let 
	\[\mathcal{A}_0 = \{(i,t), \, i\in\mathbb{N}, 0\leq t\leq D_i\}.
\]
A point $(i,t)$ of $\mathcal{A}_0$ should be thought of as representing the block $\pi_{(i)}(t)$. We equip $\mathcal{A}_0$ with the pseudo-distance function $d$ defined such: for all $x=(i,t)$ and $y=(j,s)$ in $\mathcal{A}_0$,
	\[d(x,y)= t+s-2\min(D_{\{i,j\}},s,t).
\]
(equivalently, $d(x,y)= t+s-2D_{\{i,j\}}$ if $D_{\{i,j\}}\leq s,t$ and $d(x,y)=|t-s|$ otherwise.) Let us check that $d$ verifies the four-point inequality (which in particular, implies the triangle inequality). Let $x=(i,t)$, $y=(j,s)$, $u=(k,a)$, $v=(l,b)$ be in $\mathcal{A}_0$, we want to check that, out of $\min(D_{\{i,j\}},t,s) + \min(D_{\{k,l\}},a,b)$,  $\min(D_{\{i,k\}},t,a)+\min(D_{\{j,l\}},s,b)$ and $\min(D_{\{i,l\}},t,b)+\min(D_{\{j,k\}},s,a)$, two are equal and the third one is bigger. Now, there are, up to reordering, two possible cases: either $i$ and $j$ split from $k$ and $l$ at the same time or $i$ splits from $\{j,k,l\}$ at time $t_1\geq0$, then splits $j$ from $\{k,l\}$ at time $t_2\geq t_1$ and then splits $k$ from $l$ at time $t_3 \geq t_2$. After distinguishing these two cases, the problem can be brute-forced through.

Now we want to get an actual metric space out of $\mathcal{A}_0$: this is done by identifying two points of $\mathcal{A}_0$ which represent the same block. More precisely, let us define an equivalence relation $\sim$ on $\mathcal{A}_0$ by saying that, for every pair of points $(i,t)$ and $(j,s)$, $(i,t)\sim(j,s)$ if and only if $d\big((i,t),(j,s)\big)=0$ (which means that $s=t$ and that $i\sim_{\Pi(t^-)}j$). Then we let $\mathcal{A}$ we the quotient set of $\mathcal{A}_0$ by this relation:
	\[\mathcal{A} = \mathcal{A}_0/\sim \; .
\]
The pseudo-metric $d$ passes through the quotient and becomes an actual metric. Even better, the four-point condition also passes through the quotient, and $\mathcal{A}$ is trivially path-connected: every point $(i,t)$ has a simple path connecting it to $(i,0)\sim(1,0)$, namely the path $(i,s)_{0\leq s\leq t}$. Therefore, $\mathcal{A}$ is an $\R$-tree, and we will root it at $\rho=(1,0)$. Finally, we let $\T$ be the metric completion of $\mathcal{A}$. It is still a tree, since the four-point condition and connectedness easily pass over to completions.

It is simple to see that $\T$ does satisfy assumptions $(i)$, $(ii)$, $(iii)$ by chosing $Q_i=(i,D_i)$ for all $i$: $(i)$ and $(iii)$ are immediate, and $(ii)$ comes from the definition of $d$, which is such that for all $i$ and $j$, $d\big((i,D_i),(j,D_j)\big)=D_i+D_j-2D_{i,j}$.
\end{proof}


The natural order on $\T$ is simply described in terms of $\pi$:
\begin{prop} Let $(i,t)$ and $(j,s)$ be in $\mathcal{A}$. We have $(i,t)\leq (j,s)$ if and only if $t\leq s$ and $j$ and $i$ are in the same block of $\pi(t^-)$.
\end{prop}
\begin{proof} By definition, we have $(i,t)\leq (j,s)$ if and only if $(i,t)$ is on the segment joining the root and $(j,s)$. Since this segment is none other than $(j,u)_{u\leq s}$, this means that $(i,t)\leq (j,s)$ if and only if $t\leq s$ and $(i,t)\sim(j,t)$. Now, recall that $(i,t)\sim(j,t)$ if and only if $2t - 2\min(D_{i,j},t)=0$, i.e. if and only if $t\leq D_{i,j}$, and then notice that this last equation is equivalent to the fact that $i$ and $j$ are in the same block of $\pi(t^-).$ This ends the proof.
\end{proof}

The genealogy tree has a canonical measure to go with it, at least under a few conditions: assume that $\T$ is compact, that, for all times $t$, $\pi(t^-)$ has asymptotic frequencies, and that, for all $i$, the function $t\mapsto |\pi_{(i)}(t^-)|$ (the asymptotic frequency of the block of $\pi(t^-)$ containing $i$) is left-continuous (this is not necessarily true, but when it is true it implies that the notation is in fact not ambiguous). Then Proposition \ref{01} tells us that there exists a unique measure $\mu$ on $\T$ such that, for all $(i,t)\in\T, \mu(T_{i,t})=|\pi_{(i)}(t^-)|$.

\subsection{A family of subtrees, an embedding in $\ell^1$, and measurability}

\begin{prop}\label{12}
There exists a measurable function $\TREE: \D \to \TT_W$ such that, when $\Pi^{\alpha}$ is a self-similar fragmentation process, $\TREE(\Pi^{\alpha})$ is the genealogy tree $\T$ of $\Pi^{\alpha}$ equipped with its natural measure.
\end{prop}

This will be proven by providing an embedding of $\T$ in the space $\ell^1$ of summable real-valued sequences:
	\[\ell^1=\{x=(x_i)_{i\in\N}; \sum_{i=1}^{\infty}|x_i| <\infty\}
\]
and approximating $\T$ by a family of simpler subtrees. For any finite block $B$, let $\T_B$ be the tree obtained just as before but limiting $\pi$ to the integers which are in $B$:
	\[\T_B = \{(i,t), \, i\in B, 0\leq t\leq D_i\}/\sim \; .
\]
Every $\T_B$ is easily seen to be an $\R$-tree since it is a path-connected subset of $\T$, and is also easily seen to be compact since it is just a finite union of segments. Also note that one can completely describe $\T_B$ by saying that it is the reunion of segments indexed by $B$, such that the segment indexed by integer $i$ has length $D_i$ and two segments indexed by integers $i$ and $j$ split at height $D_{\{i,j\}}$.

The tree $\T_B$ is also equipped with a measure called $\mu_B$, which we define by

  \[\mu_{B}=\frac{1}{\# B}\sum_{i\in B} \delta_{Q_i}.
\]

\begin{figure}[ht]
\centering

{\includegraphics{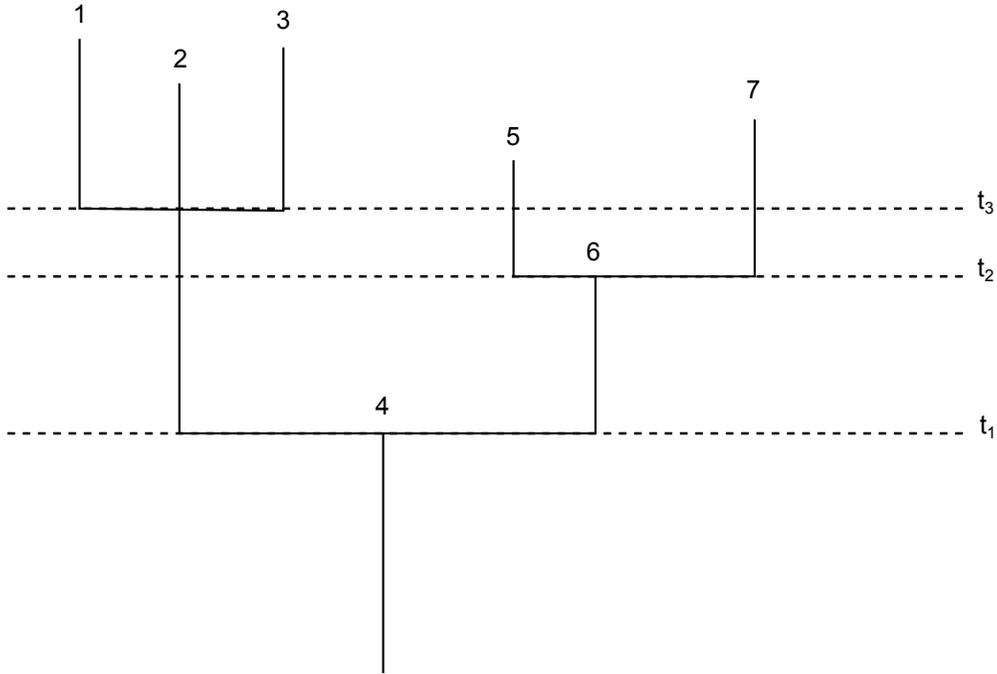}}
\caption{A representation of $\T_{[7]}.$ Here, $D_{[7]}=t_1$, $D_{\{5,6,7\}}=t_2$ and $D_{\{1,2,3\}}=t_3$}
\end{figure}

Let us provide a simultaneous embedding of $\T_B$ in $l^1$ for all $B$ such that, if $B\subset C$, $\T_B\subset \T_C$. It should be clear that the crucial part of this embedding will be the points $(i,D_i)$ for integers $i$. We are therefore going first to build points $Q_i$ in $l^1$ which will be the images of all the $(i,D_i)$ through our embedding. We use a method inspired by Aldous' "stick-breaking" method used in \cite{A3}: the path from $0$ to $Q_i$ will be followed by "increasing the coordinate corresponding to the smallest integer in the block containing $i$".

More precisely, let $i\in B$ and $j\leq i$, we let $Q_i^j$ be the total time for which $j$ has been the smallest element of the block of $\pi$ containing $i$. If $1<j<i$, this can be written as
	\[Q_i^j= \underset{k\leq j}\max \, D_{\{k,i\}}-\underset{k\leq j-1}\max\, D_{\{k,i\}},
\]
while $Q_i^1=D_{\{1,i\}}$ and $Q_i^i=D_i - \underset{k\leq i-1}\max \, D_{\{k,i\}}$. By then letting
	\[Q_i=(Q^1_i,Q^2_i,\ldots,Q^i_i,0,0,\ldots),
\]
we have defined a point $Q_i$ which has norm $D_i$.

Now that we have constructed what are going to be the endpoints of $T_B$, we need to explicit the paths from $0$ to those endpoints. Let, for every $n$, $p_n$ be the natural projection of $\ell^1$ onto $\R^n\times\{(0,0,\ldots)\}$ which sets all coordinates after the first $n$ ones to $0$. Then, for $x\in \ell^1$, we define the specific path
	\[\llbracket 0,x\rrbracket=\cup_{n=0}^{\infty} [p_n(x),p_{n+1}(x)]
\]
(where, for two points $a$ and $b$, $[a,b]$ is the line segment between those two points).

We will now prove that the set $\cup_{i\in B} \llbracket 0,Q_i\rrbracket,$ equipped with the metric inherited from the $\ell^1$ norm, is isometric to $\T_B$. We only need to check that, for integers $i$ and $j$, the segments $\llbracket 0,Q_i\rrbracket$ and $\llbracket 0,Q_j\rrbracket$ coincide until time $D_{\{i,j\}}$ and never cross afterwards. Notice that, for integers $k$ such that $D_{\{k,i\}}<D_{\{i,j\}}$, we have $D_{\{k,i\}}=D_{\{k,j\}}$. Then by construction, the two segments do indeed coincide until time $D_{\{i,j\}}$. After this time, the smallest element of the blocks containing $i$ and $j$ will always be different, so the paths will always follow different coordinates, and therefore they will never cross again.

\begin{lemma} For every finite block $B$, there exists a measurable function $\TREE_B: \D \to \TT_W$ such that, when $\pi$ is a decreasing element of $\D$ such that $D_i$ is finite for all $i$, $\TREE_B(\pi)$ is the tree $\T_B$ defined above, equipped with the measure $\mu_B$.
\end{lemma}

\begin{proof} Note that, since the set of decreasing functions in $\D$ is measurable and all the $D_i$ all also measurable functions, we only need to define $\TREE_B$ in our case of interest, and can set it to be any measurable function otherwise.

We will now in fact prove that $\T_B$ is a measurable function of $\pi$ as a compact subset of $\ell^1$ with the Hausdorff metric. First notice that, for all $i$, $Q_i$ is a measurable function of $\pi$ (this is because all of its coordinates are themselves measurable). Note then that the map $x\to \llbracket 0,x\rrbracket$ from $\ell^1$ to the set of its compact subsets is a $1$-Lipschitz continuous function of $x$. This follows from the fact that, for every $n\in\N$, and given two points $x=(x_i)_{i\in\N}$ and $y=(y_i)_{i\in\N}$, 
\begin{align*}
d_H(\{p_n(x)+tx_{n+1}e_{n+1}, t\in[0,1]\},\{p_n(y)+ty_{n+1}e_{n+1}, t\in[0,1]\})&\leq ||p_{n+1}(x-y)|| \\
                                 &\leq ||x-y||.
\end{align*}
 Then finally notice that the union operator is continuous for the Hausdorff distance. Combining these three facts, one gets that $\T_B=\underset{i\in B}\cup \llbracket 0,Q_i \rrbracket$ is indeed a measurable function of $\pi$.

The fact that $\mu_B$ is also a measurable function of $\pi$ is immediate since all the $Q_i$ are measurable.

\end{proof}

\begin{lemma}\label{13} For all $t>0$ and $\epsilon>0$, let $N_t^{\epsilon}$ be the number of blocks of $\pi(t)$ which are not completely reduced to singletons by time $t+\epsilon$. If, for any choice of $t$ and $\epsilon$, $N_t^{\epsilon}$ is almost surely finite, then the sequence $(\T_{[n]})_{n\in\N}$ is almost surely Cauchy for $d_{l^1,H}$, and the limit is isometric to $\T$. In particular, $\T$ is compact.
\end{lemma}

\begin{proof} We first want to show that the points $(Q_i)_{i\in\N}$ are \emph{tight} in the sense that for every $\epsilon>0$, there exists an integer $n$ such that any point $Q_j$ is within distance $\epsilon$ of a certain $Q_i$ with $i\leq n$. The proof of this is in essentially the same as the second half of the proof of Lemma 5 in \cite{HM04}, so we will not burden ourselves with the details here. The main idea is that, for any integer $l$, all the points $Q_i$ with $i$ such that $ht(Q_i)\in(l\epsilon,(l+1)\epsilon]$ can be covered by a finite number of balls centered on points of height belonging to $((l-1)\epsilon,l\epsilon]$ because of our assumption.

From this, it is easy to see that the sequence $(\T_{[n]})_{n\in\N}$ is Cauchy. Let $\epsilon>0$, we take $n$ just as in earlier, and $m\geq n$. Then we have 
	\[d_{\ell^1,H}(\T_{[n]},\T_{[m]})\leq \underset{n+1\leq i\leq m}\max\Big(d(Q_i,\T_{[n]})\Big)\leq \epsilon.
\]
However, since our sequence is increasing, the limit has no choice but to be the completion of their union. By the uniqueness property of the genealogy tree, this limit is $\T$.
\end{proof}

\begin{lemma} The process $\Pi^{\alpha}$ satisfies the hypothesis of Lemma \ref{13}.
\end{lemma}
\begin{proof} Once again, we refer to \cite{HM04}, where this is proved in the first half of Lemma 5. The fact that we are restricted to conservative fragmentations in \cite{HM04} does not change the details of the computations.
\end{proof}

Thus we have in particular proven that the genealogy tree of $\Pi^{\alpha}$ is compact. Let us now turn to the convergence of the measures $\mu_B$ to the measure on the genealogy tree.

\begin{lemma}\label{14} Assume that $\T$ is compact, that, for all $t$, all the blocks of $\pi(t^-)$ and $\pi(t)$ have asymptotic frequencies, and that, for all $i$, the function $t\mapsto |\pi_{(i)}(t^-)|$ (the asymptotic frequency of the block of $\pi(t^-)$ containing $i$) is left-continuous. Then the sequence $(\mu_{[n]})_{n\in\N}$ of measures on $\T$ converges to $\mu$.

\end{lemma}

\begin{proof} Since $\T$ is compact, Prokhorov's theorem assures us that a subsequence of $(\mu_{[n]})_{n\in\N}$ converges, and we will call its limit $\mu'$. Use of the portmanteau theorem (see \cite{Bill}) will show that $\mu'$ must be equal to $\mu$. Let us introduce the notation $\T_{(i,t^+)}=\cup_{s>t}\T_{(i,s)}$ for $(i,t)\in\T$ (note that this is a sub-tree of $\T$ with its root removed), we will show that $\mu'(\T_{(i,t^+)})=|\pi_{(i)}(t)|$ and $\mu'(\T_{(i,t)})=|\pi_{(i)}(t^-)|$, and uniqueness in Proposition \ref{01} will conclude. Notice that, for all $n$, by definition of $\mu_{[n]}$, we have $\mu_{[n]}(\T_{(i,t)})=\frac{1}{n} \# \big(\pi_{(i)}(t^-)\cap [n]\big)$ and $\mu_{[n]}(\T_{(i,t^+)})=\frac{1}{n}\#\big(\pi_{(i)}(t)\cap [n]\big)$ and, by definition of the asymptotic frequency of a block, these do indeed converge to $|\pi_{(i)}(t^-)|$ and $|\pi_{(i)}(t)|.$ Since $\T_{(i,t)}$ is closed in $\T$ and $\T_{(i,t^+)}$ is open in $\T$, the portmanteau theorem tells us that $\mu'(\T_{(i,t^+)})\geq|\pi_{(i)}(t)|$ and $\mu'(\T_{(i,t)})\leq|\pi_{(i)}(t^-)|.$ By writing out  
	\[\T_{(i,t)}=\cap_{n\in\N}\T_{(i,(t-\frac{1}{n})^+)},
\]
we then get
	\[\mu'(\T_{(i,t)})\geq \underset{s\to t^-}\lim \mu'(\T_{(i,s^+)})\geq \underset{s\to t^-}\lim |\pi_{(i)}(s)|\geq |\pi_{(i)}(t^-)|.
\]

Thus $\mu'(\T_{(i,t)})=|\pi_{(i)}(t^-)|$ for all choices of $i$ and $t$, and Proposition \ref{01} shows that $\mu'=\mu$. This ends the proof of the lemma.
\end{proof}
Note that, if we assume that $|\pi_{(i)}(t)|$ is right-continuous in $t$ for all $i$, a similar argument would show that $\mu(\T_{(i,t^+)})=|\pi_{(i)}(t)|$ for all $i$ and $t$. 
\medskip

Combining everything we have done so far shows that, under a few conditions, $(\T_{[n]},\mu_{[n]})$ converges as $n$ goes to infinity to $(\T,\mu)$ in the $d_{GHP}$ sense. We can now define the function $\TREE$ which was announced. The set of decreasing elements $\pi$ of $\D$ such that the sequence $(\T_{[n]},\mu_{[n]})_{n\in\N}$ converges is measurable since every element of that sequence is measurable. Outside of this set, $\TREE$ can have any fixed value. Inside of this set, we let $\TREE$ be the aforementioned limit. Since, in the case of the fragmentation process $\Pi^{\alpha}$, the conditions for convergence are met, $\TREE(\Pi^{\alpha})$ is indeed the genealogy tree of $\Pi^{\alpha}$.


\subsection{Proof of Theorem \ref{11}}

We let $(\T,\mu)=\TREE(\Pi^{\alpha})$ and want to show that it is indeed an $\alpha$-self-similar tree as defined earlier. Let $t\geq0$, and let $\pi=\Pi^{\alpha}(t)$. For all $i\in\N$ such that $\pi_i$ is not a singleton, let $\T_i(t)$ be the connected component of $\{x\in\T,ht(x)>t\}$ containing $Q_j$ for all $j\in\pi_i$, and let $x_i=(j,t)$ for any such $j$. We let also $f_i$ be any bijection: $\N\to\pi_i$
and $\Psi_i$ be the process defined by $\Psi_i(s)=f_i \big(\Pi^{\alpha}(t+|\pi_i|^{-\alpha}s)\cap\pi_i\big)$ for $s\geq0$. Let us show that, for all $i$, $(|\pi_i|^{\alpha} (\T_i(t)\cup\{x_i\}),\mu_{\T_i(t)})=\TREE(\Psi_i)$. First, $\T_i(t)\cup\{x_i\}$ is compact since it is a closed subset of $\T$. The death points of $\Psi_i$, which we will call $(Q'_j)_{j\in\N}$ are easily found: for all $j\in\N$, we let $Q'_j=Q_{f(j)}$, it is in $\T_i$ since $f(j)$ is in $\pi_i$. By the definition of $\Psi$, these points have the right distances between them. Similarly, the measure is the expected one: for $(j,s) \in \T_i$, we have $\mu(\T_{j,s})=|\Pi^{\alpha}_{(j)}(s^-)|=|\pi_i||\Psi_{(j)}((s-t)^-)|$, which is what was expected.

From the equation $(|\pi_i|^{\alpha} (\T_i(t)\cup\{x_i\}),\mu_{\T_i(t)})=\TREE(\Psi_i)$ will come the $\alpha$-self-simimlarity property. Recall that
	\[\mathcal{G}_t=\sigma(\Pi^{\alpha}(s),s\leq t)
\]
and let
	\[\mathcal{C}_t=\sigma(|\Pi^{\alpha}_i(s)|,s\leq t,i\in\N)=\sigma(\mu(\T_i(s)),s\leq t,i\in\N).
\]
We know that, conditionally on $\mathcal{F}_t$, the law of the sequence $(\Psi_i)_{i\in\N}$ is that of a sequence of independent copies of $\Pi^{\alpha}$. Since this law is fixed and $\mathcal{C}_t\subset\mathcal{F}_t$, we deduce that this is also the law of the sequence conditionally on $\mathcal{C}_t$. Applying $\TREE$ then says that, conditionally on $\mathcal{C}_t$, the $(|\pi_i|^{\alpha} (\T_i(t)\cup\{x_i\}),\mu_{\T_i(t)})_{i\in\N}$ are mutually independent and have the same law as $(\T,\mu)$ for all choices of $i\in\N$.

Finally, we need to check that the fragmentation process derived from $(\T,\mu)$ has the same law as $\Pi^{\alpha}$.
Let $(P_i)_{i\in\N}$ be an exchangeable sequence of $\T$-valued variables directed by $\mu$. The partition-valued process $\Pi_{\T}$ defined in Proposition \ref{1more} is an $\alpha$-self-similar fragmentation process. To check that it has the same law as $\Pi^{\alpha}$, one only needs to check that it has a.s. the same asymptotic frequencies as $\Pi^{\alpha}$. Indeed, Bertoin's Poissonian construction shows that the asymptotic frequencies of a fragmentation process determine $\alpha$, $c$ and $\nu$.
Let $t\geq0$, take any non-singleton block $B$ of $\Pi_{\T}(t)$, and let $C$ be the connected component of $\{x\in \T, ht(x)>t\}$ containing $P_i$ for all $i\in B$. By the law of large numbers, we have $|B|=\mu(C)$ almost surely. Thus the nonzero asymptotic frequencies of the blocks of $\Pi_{\T}(t)$ are the $\mu$-masses of the connected components of $\{x\in \T, ht(x)>t\}$, which are of course the asymptotic frequencies of the blocks of $\Pi^{\alpha}(t)$. We then get this equality for all $t$ almost surely by first looking only at rational times and then using right-continuity.  \qed

\subsection{Leaves of the fragmentation tree}
\begin{prop} There are three kinds of points in $\T=\TREE(\Pi^{\alpha})$:

-skeleton points, which are of the form $(i,t)$ with $t<D_i.$

-"dead" leaves, which come from the sudden total fragmentation of a block: they are the points $(i,D_i)$ such that $|\Pi^{\alpha}_{(i)}(D_i^-)|\neq 0$ but $\Pi^{\alpha}(D_i)\cap \Pi^{\alpha}_{(i)}(D_i^-)$ is only made of singletons. These are the leaves which are atoms of $\mu$.

-"proper" leaves, which are either of the form $(i,D_i)$ such that $|\Pi^{\alpha}_{(i)}(D_i^-)|=0$ or which are limits of sequences of the form $(i_n,t_n)_{n\in\N}$ such $|\Pi^{\alpha}_{(i_n)}(t_n)|$ tends to $0$ as $n$ goes to infinity.
\end{prop}

Note that, if $\nu$ is conservative and the erosion coefficient is zero, then there are no dead leaves: all the processes $(|\Pi^{\alpha}_{(i)}(t)|)_{t< D_i}$ continuously tend to $0$. On the other hand, if $\nu$ is not conservative or if there is some erosion, then all the $(i,D_i)$ are either skeleton points or dead leaves, and all the proper leaves can only be obtained by taking limits.

Recall the construction of the $\alpha$-self-similar fragmentation process through a homogeneous fragmentation process, which we will call $\Pi$, and the time changes $\tau_i$ defined, for all $i$ and $t$ by $\tau_i(t) = \inf \{u, \int_0^u |\Pi_{(i)}(r)|^{-\alpha}dr >t\}$.  Notice also that if $t> D_i$, $\tau_{i}(t)=\infty$.

\begin{prop}\label{leaf} Let $(i_n,t_n)_{n\in\N}$ be a strictly increasing sequence of points of the skeleton of $\T$, which converges in $\T$. The following three points are equivalent:

$(i)$ The limit of the sequence $(i_n,t_n)_{n\in\N}$ is a proper leaf of $\T$.

$(ii)$ $|\Pi^{\alpha}_{(i_n)}(t_n^-)|$ goes to $0$ as $n$ tends to infinity.

$(iii)$ $\tau_{i_n}(t_n)$ goes to infinity as $n$ tends to infinity.
\end{prop}

\begin{proof} $(i)$ and $(ii)$ are equivalent by the definition of a proper leaf. The fact that $(iii)$ implies $(ii)$ is simple. Note that, for every pair $(i,t)$ which is in $\T$, we have by definition $t\geq \tau_i(t)|\Pi_{(i)}^{\alpha}(t^-)|^{|\alpha|}$. Since $\T$ is bounded, the product $\tau_{i_n}(t_n)|\Pi_{(i_n)}^{\alpha}(t_n^-)|^{|\alpha|}$ must stay bounded. Thus, if one factor tends to infinity, the other one must tend to $0$. Finally, let us show that if $(iii)$ does not hold, then $(ii)$ also does not. Assume that $\tau_{(i_n)}(t_n)$ converges to a finite number $l$. Now we know that, because of the Poissonian way that $\Pi$ is constructed, $\underset{n\in\N}\cap \Pi_{(i_n)}(\tau_{i_n}(t_n))$ is a block of $\Pi(l^-)$. Let $i$ be in this block, we can now assume that $i_n=i$ for all $n$, and that $t_n$ converges to $D_i$ as $n$ goes to infinity, with $\tau_{(i)}(D_i)=l$. The limit of $|\Pi^{\alpha}_{(i)}(t_n^-)|$ as n tends to infinity is then $|\Pi_{(i)}(l^-)|$, which is nonzero because the subordinator $-\log(|\Pi_{(i)}(t)|)_{t\geq0}$ cannot continuously reach infinity in finite time.
\end{proof}

General leaves of $\T$ can also be described the following way: let $L$ be a leaf. For all $t<ht(L)$, $L$ has a unique ancestor with height $t$. This ancestor is a skeleton point of the form $(j,t)$ with $j\in\N$. Letting $i_L(t)$ be the smallest element of $\Pi_{(j)}(t^-)$, then $(i_L(t),t)_{t<ht(L)}$ is a kind of canonical description of the path going to $L$ and uniquely determines $L$.

\section{Malthusian fragmentations, martingales, and applications}
In order to study the fractal structure of $\T$ in detail, we will need some additional assumptions on $c$ and $\nu$: we turn to the \emph{Malthusian} setting which was first introduced by Bertoin and Gnedin in \cite{BG}, albeit in a very different environment, since they were interested in fragmentations with a nonnegative index of self-similarity.
\subsection{Malthusian hypotheses and additive martingales}
In this section, we will mostly be concerned with homogeneous fragmentations: $(\Pi(t))_{t\geq0}$ is the $(\nu,0,c)$-fragmentation process derived from a point process $(\Delta_t,k_t)_{t\geq0}$, with  dislocation measure $\nu$ and erosion coefficient $c$, and $(\mathcal{F}_t)_{t\geq0}$ is the canonical filtration of the point process.

We first start with a few analytical preliminaries. For convenience's sake, we will do a translation of the variable $p$ of the Laplace exponent $\phi$ defined in Section 2.1.4:

\begin{lemma} For all real $p$, let $\psi(p)=\phi(p-1)=cp+\int_{\s} (1-\sum_i s_i^p) d\nu(\mathbf{s})$. Then $\psi(p)\in[-\infty,+\infty)$, and this function is strictly increasing and concave on the set where it is finite.
\end{lemma}
\begin{proof} The only difficult point here is to prove for all real $p$ that $\psi(p)\in[-\infty,+\infty)$. In other words, we want to give an upper bound to $1-\sum_i s_i^p$ which is integrable with respect to $\nu$. This bound is $1-s_1^p$. Indeed, by letting $C_p=\sup_{x\in[0,1[} \frac{1-x^p}{1-x}$ (which is finite), we have $1-s_1^p\leq C_p(1-s_1)$, and $1-s_1$ is integrable by assumption.
\end{proof}

Note that, even for negative $p$, as soon as $\phi(p)>-\infty$, we have, for all $t$, 
	\[E\big[|\Pi_{(1)}(t)|^p\mathbf{1}_{\{|\Pi_{(1)}(t)|>0\}}\big]=e^{-t\phi(p)}.
\]
This follows from the description of the Lévy measure of the subordinator $\xi_t=-\log{|\Pi_{(1)}(t)|}$ (see \cite{B01}, Theorem 3).

\begin{defi} We say that the pair $(c,\nu)$ is \emph{Malthusian} if there exists a strictly positive number $p^*$ (which is necessarily unique), called the \emph{Malthusian exponent} such that
	\[\phi(p^*-1)=\psi(p^*)=cp^* + \int_{\s} \big(1-\sum_{i=1}^\infty s_i^{p^*}\big) d\nu(\mathbf{s}) = 0.
\]
\end{defi}

The typical example of pairs $(c,\nu)$ with a Malthusian exponent are \emph{conservative} fragmentations, where $c=0$ and $\sum_i s_i=1$ $\nu$-almost everywhere. In that case, the Malthusian exponent is simply $1$.
Note that assumption $\mathbf{(H)}$ defined in the introduction implies the existence of the Malthusian exponent, since $\psi(1)\geq0$ for all choices of $\nu$ and $c$.

We assume from now on the existence of the Malthusian exponent.
\begin{prop} For all $i\in\N$ and $t\geq0$, we let
	\[M_{i,t}(s)= \sum_{j=1}^{\infty}  |\Pi_j(t+s)\cap\Pi_{(i)}(t)|^{p^*}.
\]
The process $(M_{i,t}(s))_{s\geq0}$ is a càdlàg martingale with respect to the filtration $(\mathcal{F}_{t+s})_{s\geq0}$.

We let $M_{1,0}(t)=M(t)$ for all $t$.
\end{prop}

\begin{proof}
Let us first notice that, as a consequence of the fragmentation property, for every $(i,t)$, the process $(M_{i,t}(s))_{s\geq0}$ has the same law as a copy of the process $(M(s))_{s\geq0}$ which is independent of $\mathcal{F}_t$, multiplied by $|\Pi_{(i)}(t)|^{p^*}$ (which is an $\mathcal{F}_t$-measurable variable). Thus, we only need to prove the martingale property for $(M(s))_{s\geq0}$. Recall that, given $\pi\in\p_{\N}$, $\rep(\pi)$ is the set of integers which are the smallest element of the block of $\pi$ containing them and let $t\geq0$ and $s\geq0$, we have, 
\begin{align*}
E[M(t+s)\, | \, \mathcal{F}_s]&=E\left[\sum_{i\in \rep(\Pi(s))} M_{i,s}(t)\, | \, \mathcal{F}_s\right] \\
                              &=\sum_{i\in\N}E[\mathbf{1}_{\{i\in \rep(\Pi(s))\}} M_{i,s}(t)\, | \, \mathcal{F}_s] \\
                              &=\sum_{i\in\N}\mathbf{1}_{\{i\in \rep(\Pi(s))\}}|\Pi_{(i)}(s)|^{p^*} E[M(t)]\\
                              &=\sum_{i\in \rep(\Pi(s))} |\Pi_{(i)}(s)|^{p^*} E[M(t)] \\
                              &=M(s) E[M(t)].
\end{align*}

Thus we only need to show that $E[M(t)]=1$ for all $t$ and our proof will be complete. To do this, one uses the fact that, since $\Pi(t)$ is an exchangeable partition, the asymptotic frequency of the block containing $1$ is a size-biased pick from the asymptotic frequencies of all the blocks. This tells us that
\begin{align*}
E\left[\sum_i |\Pi_i(t)|^{p*}\right]&=E\left[\sum_i |\Pi_i(t)||\Pi_i(t)|^{p*-1}\mathbf{1}_{|\Pi_{(i)}(t)|\neq0}\right] \\
                       &=E[|\Pi_{(1)}(t)|^{p^*-1}\mathbf{1}_{\{|\Pi_{(i)}(t)|\neq0\}}] \\
                       &=\exp[-t\phi(p^*-1)] \\
                       &=1.
\end{align*}

We refer to \cite{BR} for the proof that $(M(t))_{t\geq0}$ is càdlàg (it is assumed in \cite{BR} that $c=0$ and that $\nu$ is conservative but these assumptions have no effect on the proof).
\end{proof}

Since these martingales are nonnegative, they all converge almost surely. For integer $i$ and real $t$, we will call $W_{i,t}$ the limit of the martingale $M_{i,t}$ on the event where this martingale converges. We also write $W$ instead of $W_{1,0}$ for simplicity. Our goal is now to investigate these limits. To this effect, let us introduce a family of integrability conditions indexed by parameters $q>1$, we let $\mathbf{(M_q)}$ be the assumption that
	\[\int_{\s}\left|1-\sum_{i=1}^{\infty}s_i^{p^*}\right|^q d\nu(\mathbf{s}) <\infty.
\]
We will assume through the rest of this section that there exists some $q>1$ such that $\mathbf{(M_q)}$ holds.

The following is a generalization of Theorem 1.1 and Proposition 1.5 of \cite{Ber} which were restricted to the case where $\nu$ has finite total mass.
\begin{prop}\label{Lq} Assume $\mathbf{(M_q)}$ for some $q>1$. Then the martingale $(M(t))_{t\geq0}$ converges to $W$ in $L^q$.
\end{prop}

\begin{proof} We will first show that the martingale $(M(t))_{t\geq0}$ is purely discontinuous in the sense of \cite{DM}, which we will do by proving that it has finite variation on any bounded interval $[0,T]$ with $T>0$. To this effect, write, for all $t$, $M(t)=\e^{-cp^*t}\sum_i (X_i(t))^{p^*}$ where the $(X_i(t))_{i\in\N}$ are the sizes of the blocks of a homogeneous fragmentation with dislocation measure $\nu$, but no erosion. Since the product of a bounded nonincreasing function with a bounded function of finite variation has finite variation, we only need to check that $t\mapsto \sum_i X_i(t)^{p^*}$ has finite variation on $[0,T]$. Since this function is just a sum of jumps, its total variation is equal to the sum of the absolute values of these jumps. Thus we want to show that $|\underset{t\leq T}\sum \underset{i}\sum(X_i(t))^{p^*}-(X_i(t^-))^{p^*}|$ is finite. This sum is equal to $\underset{t\leq T}\sum e^{cp^*t} |M(t)-M(t^-)|$, which is bounded above by $e^{cp^*T}\underset{t\leq T}\sum |M(t)-M(t^-)|$. We will not show the finiteness of this sum, which is done by computing its expectation similarly to our next computation.

Knowing that the martingale is purely discontinuous, according to \cite{Lep} (at the bottom of page 299), to show that the martingale is bounded in $L^q$, one only needs to show that the sum of the $q-th$ powers of its jumps is also bounded in $L^q$, i.e. that
	\[E\big[\sum_t |M(t)-M(t^-)|^q\big]<\infty.
\]

This expected value can be computed with the Master formula for Poisson point processes ( see\cite{RY}, page 475). Recall the construction of $\Pi$ through a family of Poisson point processes $((\Delta^k(t))_{t\geq0})_{k\in\N}$: for $t$ and $k$ such that there is an atom $\Delta^k(t)$, the $k$-th block of $\Pi(t^-)$ is replaced by its intersection with $\Delta^k(t)$. We then have
\begin{align*}
E\left[\sum_{t\geq0} |M(t)-M(t^-)|^q\right]&=E\left[\sum_{t\geq0}|\Pi_{k_t}(t-)|^{qp^*}\left(|1-\sum_{i=1}^{\infty}|\Delta^k_i(t)|^{p^*}|\right)^q\right] \\
                                &=E\left[\sum_{k=1}^{\infty}\sum_{t\geq 0} |\Pi_k(t-)|^{qp^*}\left(|1-\sum_{i=1}^{\infty}|\Delta^k_i(t)|^{p^*}|\right)^{q}\right]\\
                                &=E\left[\int_0^\infty \sum_k |\Pi_k(t-)|^{qp^*} dt\right] \int_{\s} |1-\sum_i s_i^{p^*}|^q d\nu(\mathbf{s}) \\
                                &=\int_0^\infty \e^{-t\psi(qp^*)} dt \int_{\s} |1-\sum_i s_i^{p^*}|^q d\nu(\mathbf{s}).                               
\end{align*}
Since $qp^*>p^*$, we have $\psi(qp^*)>0$ and thus the expectation is finite. 
\end{proof}

\begin{prop} Assume that $E[W]=1$ (which is equivalent to assuming that the martingale $M((t))_{t\geq0}$ converges in $L^1$). Then, almost surely, if $\Pi$ does not die in finite time then $W$ is strictly positive.
\end{prop}

\begin{proof}
We discretize the problem and only look at integer times: for $n\in\N$, let $Z_n$ is the number of blocks of $\Pi(n)$ which have nonzero mass. The process $(Z_n)_{n\in\N}$ is a Galton-Watson process (possibly taking infinite values. See Appendix 2 to check that standard results stay true in this case). If it is critical or subcritical then there is nothing to say, and if it is supercritical, notice that the event $\{W=0\}$ is hereditary (in the sense that $W=0$ if and only if all the $W_{i,1}$ are also zero). This implies that the probability of the event $\{W=0\}$ is either equal to $1$ or to the probability of extinction. But since $E[W]=1$, $W$ cannot be $0$ almost surely and thus $\{W=0\}$ and the event of extinction have the same probabilities. Since $\{W=0\}$ is a subset of the event of extinction, $W$ is nonzero almost surely on nonextinction.

\end{proof}

The following proposition states the major properties of these martingale limits.

\begin{prop}\label{21} There exists an event of probability $1$ on which the following are true:

(i) For every $i$ and $t$, the martingale $M_{i,t}$ converges to $W_{i,t}$.

(ii) For every integer $i$, and any times $t$ and $s$ with $s>t$, we have
	\[W_{i,t}=\sum_{j \in \Pi_{(i)}(t)\cap \rep(\Pi(s))} W_{j,s}.
\]

(iii) For every $i$, the function $t\mapsto W_{i,t}$ is nonincreasing and right-continuous. The left-limits can be described as follows: for every $t$, we have
	\[W_{i,t^-}=\sum_{j \in \Pi_{(i)}(t^-)\cap \rep(\Pi(t))} W_{j,t}.
\]
\end{prop}

To prove this we will need the help of several lemmas. The first is an intermediate version of point $(ii)$

\begin{lemma}\label{22} For any integer $i$ and any times $t$ and $s$ such that $s>t$, there exists an event of probability $1$ on which the martingales $M_{i,t}$ and $M_{j,s}$ converge for all $j$ and we have the relation
\[W_{i,t}=\sum_{j \in \Pi_{(i)}(t)\cap \rep(\Pi(s))} W_{j,s}.
\]
\end{lemma}

\begin{proof} For clarity's sake, we are going to restrict ourselves to the case where $i=1$ and $t=0$, but the proof for the other cases is similar. We have, for all $r\geq s$,
	\[M(r)=\sum_{j\in \rep(\Pi(s))} M_{j,s}(r-s).
\]
We cannot immediately take the limits as $r$ goes to $\infty$ because we do not have any kind of dominated convergence under the sum. However, Fatou's Lemma does give us the inequality
	\[W\geq\sum_{j\in \rep(\Pi(s))} W_{j,s}.
\]
 To show that these are actually equal almost surely, we show that their expectations are equal. We know that $E[W]=1$ and that, for all $j\in\N$ and $s\geq0$, one can write $W_{j,s}=|\Pi_{(j)}(s)|^{p^*} W'_{j,s}$ where $W'_{j,s}$ is a copy of $W$ which is independent of $\mathcal{F}_s$. We thus have

\begin{align*}
E\left[\sum_{j\in \rep(\Pi(s))}W_{j,s}\right] &= E\left[\sum_{j\in \N} \mathbf{1}_{j\in \rep(\Pi(s))} |\Pi_{(j)}(s)|^{p^*} W'_{j,s}\right] \\
                                  &= \sum_{j\in \N} E[\mathbf{1}_{j\in \rep(\Pi(s))} |\Pi_{(j)}(s)|^{p^*} W'_{j,s}] \\
                                  &=\sum_{j\in \N} E[W'_{j,s}]E[\mathbf{1}_{j\in \rep(\Pi(s))} |\Pi_{(j)}(s)|^{p^*}] \\
                                  &=\sum_{j\in \N} E[\mathbf{1}_{j\in \rep(\Pi(s))} |\Pi_{(j)}(s)|^{p^*}] \\
                                  &=E[M(s)] \\
                                  &=1.
\end{align*}

\end{proof}

\begin{lemma}\label{23} For every pair of integers $i$ and $j$, let $f_{i,j}$ be a nonnegative function defined on $[0,+\infty)$. For every $i$, we let $f_i$ be the function $\sum_j f_{i,j}$, and we also let $f=\sum_i f_i$. We assume that, for every $i$ and $j$, the function $f_{i,j}$ converges at infinity to a limit called $l_{i,j}$, and we also assume that $f$ converges, its limit being $l=\sum_{i,j} l_{i,j}$. Then, for every $i$, the function $f_i$ also converges at infinity and its limit is $l_i=\sum_j l_{i,j}$.
\end{lemma}

\begin{proof} We are going to prove that $\liminf f_i = \limsup f_i =l_i$ for all $i$. Let $N$ be any integer, taking the upper limit in the relation $f\geq \sum_{i\leq N} f_i$ gives us $l\geq \sum_{i\leq N} \limsup f_i$, and by taking the limit as $N$ goes to infinity, we have $l\geq \sum_{i} \limsup f_i$. Similarly, for every $i$, the relation $f_i=\sum_j f_{i,j}$ gives us $\liminf f_i \geq \sum_j l_{i,j}$. We thus have the following chain:
	\[\sum_{i,j} l_{i,j} \leq \sum_i \liminf f_i \leq \sum_i \limsup f_i \leq \sum_{i,j} l_{i,j},
\]
and this implies that, for every $i$, $\liminf f_i = \limsup f_i =l_i.$
\end{proof}
\emph{Proof of Proposition \ref{21}:} let $t<s$ be two times and assume that the martingale $M_{j,s}$ converges for all $j$, and also assume the relation $W=\underset{j\in\rep{\Pi(s)}}\sum W_{j,s}$. Apply Lemma \ref{23} with $f(r)=M(s+r)$, $f_i(r)=\mathbf{1}_{\{i\in\rep \Pi(t)\}}M_{i,t}(r+s-t)$ and $f_{i,j}(r)=\mathbf{1}_{\{i\in\rep \Pi(t)\}}\mathbf{1}_{\{j\in\rep \Pi(s)\cap \Pi_{(i)}(t)\}}M_{j,s}(r)$. Then, for all $i$, the martingale $M_{i,t}$ does indeed converge, and point $(ii)$ of the proposition is none other than the relation $l_i=\sum_j l_{i,j}$. We also get that $W=\underset{i\in\rep{\Pi(t)}}\sum W_{i,t}$ and thus can use the same reasoning to obtain $W_{i,r}=\underset{j\in\Pi_{(i)}(r)\cap\rep\Pi(t)}\sum W_{j,t}$ for all $r<t<s$.

By Lemma \ref{22}, the assumption of the previous paragraph is true for any value of $s$ with probability $1$, we then obtain points $(i)$ and $(ii)$ by taking a sequence of values of $s$ tending to infinity.


We can turn ourselves to point $(iii)$. Fixing an integer $i$, it is clear that $t\mapsto W_{i,t}$ is nonincreasing. Right-continuity is obtained by the monotone convergence theorem, noticing that $\Pi_{(i)}(t)\cap \rep(\Pi(s))$ is the increasing union, as $u$ decreases to $t$, of sets $\Pi_{(i)}(u)\cap \rep(\Pi(s)).$ Similarly, the fact that $W_{i,t^-}=\sum_{j \in \Pi_{(i)}(t^-)\cap \rep(\Pi(t))} W_{j,t}$ is only a matter of noticing that $\Pi_{(i)}(t^-)$ is the decreasing intersection, as $u$ increases to $t$, of sets $\Pi_{(i)}(u)$ and taking the infimum on both sides of the relation $W_{i,u}=\sum_{j \in \Pi_{(i)}(u)\cap \rep(\Pi(t))} W_{j,t}$.
\qed
\smallskip

From now on we will restrict ourselves to the aforementioned almost-sure event: all the additive martingales are now assumed to converge, and the limits satisfy the natural additive properties.

\subsection{A measure on the leaves of the fragmentation tree.}

In this section we are going to assume that $E[W]=1$. We let $\T$ be the genealogy tree of the self-similar process $\Pi^{\alpha}$ and are going to use the martingale limits to define a new measure on $\T$.

\begin{theo}\label{24} On an event with probability one, there exists a unique measure $\mu^*$ on $\T$ which is fully supported by the proper leaves of $\T$ and which satisfies

	\[\forall i\in\N,t\geq0, \mu^* (\T_{(i,t^+)})=W_{i,\tau_i(t)}.
\]

where $\T_{i,t^+}$ is as defined in the proof of Lemma \ref{14}: $\T_{(i,t^+)}=\cup_{s>t}\T_{(i,s)}$.
\end{theo}

\begin{proof} This will be a natural consequence of proposition \ref{01}, and our previous study of the convergence of additive martingales. Note that, since, for all $(i,t)\in\T$, we have 
	\[\T_{(i,t)}=\underset{j\in \Pi_{(i)}(t^-)\cap \rep (\Pi(t))}\bigcup \T_{(j,t^+)},
\]
 any candidate for $\mu^*$ would then have to satisfy, for every $(i,t)$, the relation 
	\[\mu^* (\T_{(i,t)})=\underset{j\in \Pi_{(i)}(t^-)\cap \rep (\Pi(t))}\sum W_{j,\tau_i(t)}=W_{i,(\tau_i(t))^-}.
\]
We thus know to apply Proposition \ref{01} to the function $m$ defined by $m(i,t)=W_{i,(\tau_i(t))^-}$. This function is indeed decreasing and left-continuous on $\T$, and we also have, for every $(i,t)$, $m((i,t)^+)=m(i,t)$ in the sense of Section 2.2.4 (this is point $(iii)$ of Proposition \ref{21}). Thus $\mu^*$ exists and is unique, and we only now need to check that it is fully supported by the set of proper leaves of $\T$. To do this, notice first that, by Proposition \ref{leaf}, the complement of the set of proper leaves can be written as $\cup_{N\in\N} \{(i,s), \; i\in\N, \tau_i(s)\leq N\},$ and then thatt, for every integer $N$,
\[
\mu^* (\{(i,s), \; i\in\N, \tau_i(s)\leq N\}) =W - \sum_{i\in \rep (\Pi(N))} W_{i,N} =0.
\]
\end{proof}

The measure $\mu^*$ has total mass $W$, which is in general not $1$. However, having assumed that $E[W]=1$, we will be able to create some probability measures involving $\mu^*$. First, recall that to every leaf $L$ of $\T$ corresponds a family of integers $(i_L(t))_{t<ht(L)}$ such that, for all $t$, $i_L(t)$ is the smallest integer such that $(i_L(t),t)\leq L$ in $\T$.

\begin{prop} Define a probability measure $Q$ on the space $\D([0,+\infty))$ of càdlàg functions from $[0,+\infty)$ to $[0,+\infty)$ by setting, for all nonnegative measurable functionals $F: \D([0,+\infty))\to[0,+\infty)$,
	\[Q(F)=E\left[\int_{\T} F\big((|\Pi^{\alpha}_{(i_L(t))}(t)|)_{t\geq0}\big) d\mu^*(L)\right].
\]
Let $(x_t)_{t\geq0}$ be the canonical process, and let $\zeta$ be the time-change defined for all $t\geq0$ by:
	\[\zeta(t) = \inf \Big\{u, \int_0^u x_t^{\alpha}dr >u\Big\}.
\] 

Under the law $Q$, the process $(\xi_t)_{t\geq0}$ defined by $\xi_t=-\log(x_{\zeta(t)})$ for all $t\geq0$ is a subordinator whose Laplace exponent $\phi^*$ satisfies, for $p$ such that $\psi(p+p^*)$ is defined:
	\[\phi^*(p)= cp+\int_{\s} \big(\sum_i(1-s_i^p)s_i^{p^*}\big) d\nu(\mathbf{s}) =\psi(p+p^*).
\]
\end{prop}

As before, the function $\phi^*$ can be seen as defined on $\R$, in which case it takes values in $[-\infty,\infty)$.

\begin{proof} Let us first show that, given a nonnegative and measurable function $f$ on $[0,+\infty)$ and a time $t$, we have
\begin{equation}\label{eq}
Q(f(x_{\zeta(t)}))=E\left[\sum_i |\Pi_i(t)|^{p^*} f(|\Pi_i(t)|)\right].
\end{equation}
To do this, notice that we have $\Pi^{\alpha}_{(i_L(t))}(\tau_{i_L(t)}^{-1}(t))=\Pi_{(i_L(t))}(t)$. Thus, using the definition of $\mu^*$, one can change the integral with the respect to $\mu^*$ into a sum on the different blocks of $\Pi(t)$:
	\[Q(f(x_{\zeta(t)}))=E\left[\sum_{i\in \rep(\Pi(t))} W_{i,t} f(|\Pi_{(i)}(t)|)\right].
\]
Finally, with the fragmentation property, one can write, for all $t$ and $i$, $W_{i,t}=|\Pi_{(i)}(t)|^{p^*}W'_{i,t}$ where $W'_{i,t}$ is a copy of $W$ which is independent of $|\Pi(t)|$. Since $E[W]=1$, we get formula \ref{eq}.

\medskip

Applying this to the function $f$ defined by $f(x)=x^p$ gives us our moments formula:
	\[Q(\e^{-p\xi_1})=Q(x_{\zeta(1)}^p)=E\left[\sum_i |\Pi_i(1)|^{p^*+p}\right]=E[\Pi_1(1)^{p+p^*-1}]=\exp[-(\phi(p+p^*-1)].
\]

Independence and stationarity of the increments is proved the same way. Let $s<t$, $f$ be any nonnegative measurable functions on $\R$ and $G$ be any nonnegative measurable function on $\D([0,s])$. Let us apply the fragmentation property for $\Pi$ at time $s$: for $i\in \rep(\Pi(s))$, the partition of $\Pi_{(i)}(s)$ formed by the blocks of $\Pi(t)$ which are subsets of $\Pi_{(i)}(s)$ can be written as $\Pi_{(i)}(s) \cap \Pi^i(t-s)$ where $(\Pi^i(u))_{u\geq0}$ is an independent copy of $\Pi$. Thus one can write
\begin{align*}
Q[f(\frac{x_{\zeta(t)}}{x_{\zeta(s)}})&G((x_{\zeta(u)})_{u\leq s})] \\
                                      &= E\left[\sum_{i\in \rep{\Pi(s)}}  |\Pi_{(i)}(s)|^{p^*}G((|\Pi_{(i)}(u)|)_{u\leq s}) \sum_{j\in\N} W^i_j|\Pi^i_j(t-s)|^{p^*}  f(|\Pi^i_j(t-s)|)\right], 
\end{align*}  
where the $W^i_j$ are copies of $W$ independent of anything happening before time $t$, which all have expectation $1$. We thus get
\begin{align*}
Q[f(\frac{x_{\zeta(t)}}{x_{\zeta(s)}})&G((x_{\zeta(u)})_{u\leq s})]\\
                                       &=E\left[\sum_{i\in \rep{\Pi(s)}}  |\Pi_{(i)}(s)|^{p^*}G((|\Pi_{(i)}(u)|)_{u\leq s})\right]E\left[\sum_j |\Pi_i(t-s)|^{p^*}  f(|\Pi_j(t-s)|)\right],
\end{align*}
which is what we wanted.

\end{proof}

\begin{lemma}\label{25} Assume that $\nu$ integrates the quantity $\sum_i \log(s_i)s_i^{p^*}$ and let $\underline{p}=\sup \{q\in \R: \phi^*(-q)>-\infty\}$. Then if $\gamma<1+ \frac{\underline{p}}{|\alpha|}$, we have
	\[E\left[\int_{\T} ht(L)^{-\gamma}d\mu^*(L)\right] < \infty.
\]
\end{lemma}

\begin{proof} We know that the height of the leaf is equal to the death time of the fragment it marks: $ht(L) = \inf \{t,\tau_{i_L(t)}(t)=\infty\}.$ Thus we can write, using the measure $Q$
	\[E\left[\int_{\T} ht(L)^{-\gamma}d\mu^*(L)\right]=E[I^{-\gamma}],
\]
where $I=\int_{0}^{\infty} \e^{\alpha \xi_t}dt$ is the exponential functional of the subordinator $\xi$ with Laplace exponent $\phi^*$. Following the proof of Proposition 2 in \cite{BY02}, one has, if $1<\gamma<1+\frac{\underline{p}}{|\alpha|}$,
	\[E[I^{-\gamma}]=\frac{-\phi^*(-|\alpha|(\gamma-1))}{\gamma-1}E[I^{-\gamma+1}].
\]
By induction we then only need to show that $E[I^{-\gamma}]$ is finite for $\gamma\in(0,1]$, and thus only need to show that $E[I^{-1}]$ is finite. However, it is well known (see for example \cite{BY01}), $E[I^{-1}]=(\phi^*)'(0^+)=c-\int_{\s} (\sum_i \log(s_i)s_i^{p^*})d\nu(\mathbf{s})$, which is finite by assumption.
\end{proof}
The assumption that $\int_{\s} (\sum_i \log(s_i)s_i^{p^*})d\nu(\mathbf{s})$ is finite is for example verified when $\nu$ has finite total mass, and $\mathbf{(H)}$ is satisfied: pick $\delta>0$ such that $\psi(p^*-\delta)>-\infty$, then pick $K>0$ such that $|\log(x)|\leq Kx^{-\delta}$ for all $x\in(0,1]$, then $\sum_i |\log(s_i)|s_i^{p^*}\leq K(1-(1-\sum_i s_i^{p^*-\delta}))$ and is indeed integrable.

\section{Tilted probability measures and a tree with a marked leaf}

Recall that $\D$ is the space of càdlàg $\p_{\N}$-valued functions on $[0,+\infty)$, and that it is endowed with the $\sigma$-field generated by all the evaluation functions. For all $t\geq0$, let us introduce the space $\D_t$ of càdlàg functions from $[0,t]$ to $\p_{\N}$, which we endow with the product $\sigma$-field.

As was done in \cite{BR}, we are going in this section to use the additive martingale to construct a new probability measure under which our fragmentation process has a special tagged fragment such that, heuristically, for all $t$, the tagged fragment is equal to a block $\Pi_i(t)$ of $\Pi(t)$ with "probability" $|\Pi_i(t)|^{p^*}$. Tagging a fragment will be done by forcing the integer $1$ to be in it, and for this we need some additional notation. If $\pi$ is a partition of $\N$, we let $R\pi$ be its restriction to $\N'=\N\setminus\{1\}$. Partitions of $\N'$ can still be denoted as sequences of blocks ordered with increasing least elements. Given a partition $\pi$ of $\N'$ and any integer $i$, we let $H_i(\pi)$ be the partition of $\N$ obtained by inserting $1$ in the $i$-th block of $\pi$. Similarly, let us also define a way to insert the integer $1$ in a finite-time fragmentation process with state space the partitions of $\N'$. Let $i\in\N$, $t\geq0$ and let $(\pi(s))_{s\leq t}$ be a family of partitions of $\N'$. Now let $j$ be any element of $\pi_i(t)$ (if this block is empty then the choice won't matter, one can just define $H^t_i(\pi)$ to be any fixed process) and, for all $0\leq s \leq t$, let $H^t_i(\pi)(s)$ be the partition which is the exact same as $\pi(s)$, except that $1$ is added to the block containing $j$. This defines a function $H^t_i$ which maps a process taking values in $\p_{\N'}$ to processes taking value in $\p_{\N}$. What is important to note is that, if we now take $(\pi(s))_{0\leq s\leq t} \in D_t$, then the process $(H^t_i R(\pi)(s))_{0\leq s \leq t}$ is càdlàg (because the restrictions to finite subsets of $\N$ are pure-jump with finite numbers of jumps) and the map $H^t_i R$ from $\D_t$ to itself is also measurable (because, for all $s$, $H^t_i(\pi)(s)$ is a measurable function of $\pi(s)$ and $\pi(t)$).


\subsection{Tilting the measure of a single partition}
Here, we are going to work in a simple setting: we consider a random exchangeable partition of $\N$ called $\Pi$ which has a positive Malthusian exponent $p^*$, in the sense that $E\big[\sum_i |\Pi_i|^{p^*}\big]=1$. Note that this implies that $E\big[|\Pi_1|^{p^*-1}\mathbf{1}_{|\Pi_1|\neq0}\big]=1$ as well (we will omit the indicator function from now on).

\medskip

Let us define two new random partitions $\Pi^*$ and $\Pi'$ through their distributions: we let, for nonegative measurable functions $f$ on $\p_{\N}$,
	\[E^*[f(\Pi^*)]=E\left[\sum_i |R\Pi_i|^{p^*} f(H_iR\Pi)\right]
\]
and
	\[E'[f(\Pi')]=E[|\Pi_1|^{p^*-1}f(\Pi)].
\]
These relations do define probability measures because $p^*$ is the Malthusian exponent of $\Pi$, as can be checked by taking $f=1$. We now state a few properties of these distributions.

\begin{prop} 
(i) The two random partitions $\Pi^*$ and $\Pi'$ have the same distribution.

(ii) If we call $m$ the law of the asymptotic frequencies of the blocks of $\Pi$, and $m'$ the law of the asymptotic frequencies of the blocks of $\Pi'$, we have
	\[m'(d\mathbf{s})=(\sum_i s_i^{p^*}) m(d\mathbf{s}).
\]
In particular, with probability $1$, $\Pi'$ is not the partition made uniquely of singletons.

(iii) Conditionally on the asymptotic frequencies of its blocks, the law of $\Pi'$ (or $\Pi^*$) can be described as follows: the restriction of the partition to $\N'$ is built with a standard paintbox process from the law $m'$. Then, conditionally on $R\Pi'$, for every integer $i$, $1$ is inserted in the block $R\Pi'_i$ with probability $\frac{|R\Pi'_i|^{p^*}}{\sum_j |R\Pi'_j|^{p^*}}$.
\end{prop}

\begin{proof} Item $(i)$ is a simple consequence of the paintbox description of $\Pi$: we know that, conditionally on the restriction of $\Pi$ to $\N'$, the integer $1$ will be inserted in one of these blocks in a size-biased manner. Thus we get, for nonnegative measurable $f$,
	\[E'[f(\Pi')]=E[|\Pi_1|^{p^*-1}f(\Pi)]=E\big[\sum_i |R\Pi_i||R\Pi_i|^{p^*-1} f(H_iR\Pi)\big]=E^*[f(\Pi^*)].
\]

To prove $(ii)$, we just need to use the definition of the law of $\Pi'$: take any positive measurable function $f$ on $\s$, we have
	\[E'[f(|\Pi'|^{\downarrow})]=E[|\Pi_1|^{p^*-1}f(|\Pi|^{\downarrow})]=\int_{\s}(\sum_i s_i s_i^{p^*-1})f(\mathbf{s})m(d\mathbf{s})=\int_{\s}(\sum_i s_i^{p^*})f(\mathbf{s})m(d\mathbf{s}),
\]
which is all we need.

For $(iii)$, first use the definition of $\Pi^*$ to notice that its restriction to $\N'$ is exchangeable: if we take a measurable function $f$ on $\p_{\N'}$, we have
\begin{align*}
E^*[f(\sigma(R\Pi^*))]&=E\Big[(\sum_i |R\Pi_i|^{p^*}) f(\sigma (R\Pi))\Big] \\
                      &=E\Big[(\sum_i |\sigma R\Pi_i|^{p^*}) f(\sigma (R\Pi))\Big] \\
                      &=E\Big[(\sum_i |R\Pi_i|^{p^*}) f(R\Pi)\Big] \\
                      &=E^*[f(R\Pi^*)].
\end{align*}
This exchangeability and Kingman's theorem then imply that the restriction of $\Pi^*$ to $\N'$ can indeed be built with a paintbox process. Now we only need to identify which block contains $1$, that is, find the distribution of $\Pi^*$ conditionally of $R\Pi^*$. Thus, we take a nonegative measurable function $f$ on $\p_{\N}$ and another one $g$ on $\p_{\N'}$ and compute $E^*[f(\Pi^*)g(R\Pi^*)]$:
\begin{align*}
E^*[f(\Pi^*)g(R\Pi^*)]&=E\Big[\sum_i |(R\Pi)_i|^{p^*} f(H_iR\Pi)g(R\Pi)\Big] \\
                      &=E\left[\sum_j |(R\Pi)_j|^{p^*} \left(\sum_i \frac{|(R\Pi)_i|^{p^*}}{\sum_j |(R\Pi)_j|^{p^*}}f(H_iR\Pi)\right)g(R\Pi)\right] \\
                      &=E^*\left[\left(\sum_i \frac{|(R\Pi^*)_i|^{p^*}}{\sum_j |(R\Pi^*)_j|^{p^*}}f(H_iR\Pi^*)\right)g(R\Pi^*)\right].
\end{align*}

This ends the proof.
\end{proof}

\subsection{Tilting a fragmentation process}
Here we aim to generalize the previous procedure to a homogeneous exchangeable fragmentation process. Let $t\geq0$, we are going to define two random processes $(\Pi^*(s))_{s\leq t}$ and $(\Pi'(s))_{s\leq t}$, with corresponding expectation operators $E^*_t$ and $E'_t$, by letting, for measurable functions $F$ on $\D_t$,
	\[E^*_t[F((\Pi^*(s))_{s\leq t})]=E\big[\sum_i |(R\Pi(t))_i|^{p^*}F((H^t_iR\Pi(s))_{s\leq t})\big]
\]
and
	\[E'_t[F((\Pi'(s))_{s\leq t})]=E\big[|\Pi_1(t)|^{p^*-1}F\big((\Pi(s))_{s\leq t}\big)\big].
\]
For the same reason as before, these define probability measures. We then want to use Kolmogorov's consistency theorem to extend these two probability measures to $\D$. To do this we have to check that, if $u<t$ and $(\Pi^*(s))_{s\leq t}$ has law $P^*_t$, then $(\Pi^*(s))_{s\leq u}$ has law $P^*_u$, and the same for $\Pi'$. The argument is that the block of $\Pi^*(t)$ that $1$ is inserted in only matters through its ancestor at time $u$: if $i$ and $j$ are such that $(R\Pi^*(t))_j\subset (R\Pi^*(u))_i$ then $(H^t_j\Pi(s))_{s\leq u}=(H^u_i\Pi(s))_{s\leq u}$. Taking any nonnegative measurable function $F$ on $\D$, we have
\begin{align*}
E^*_t[F\big((\Pi^*(s))_{s\leq u}\big)]&=E\big[\sum_j |(R\Pi(t))_j|^{p^*}F\big((H^t_j\Pi(s))_{s\leq u}\big)\big] \\
                  &=E\big[\sum_i \sum_{j: (R\Pi^*(t))_j\subset (R\Pi^*(u))_i} |(R\Pi(t))_j|^{p^*}F\big((H^u_i\Pi(s))_{s\leq u}\big)\big] \\
                  &=E\big[\sum_i F\big((H^u_i\Pi(s))_{s\leq u}\big) \sum_{j: (R\Pi^*(t))_j\subset (R\Pi^*(u))_i} |(R\Pi(t))_j|^{p^*}\big] \\
                  &=E\big[\sum_i F\big((H^u_i\Pi(s))_{s\leq u}\big) |(R\Pi(u))_i|^{p^*}\big].
\end{align*}
The last equation comes from the martingale property of the additive martingale $M_{k,u}$ where $k$ is any integer in $(R\Pi(u))_i$.
Consistency for $\Pi'$ is a little bit simpler: it is once again a consequence of the fact that the process $(M'_t)_{t\geq0}$ defined by $M'_t=|\Pi_1(t)|^{p^*-1}\mathbf{1}_{\{|\Pi_1(t)|\neq0\}}$ for all $t$ is a martingale, which itself is an immediate consequence of the homogeneous fragmentation property.

Kolmogorov's consistency theorem then implies that there exist two random processes $(\Pi^*(t))_{t\geq0}$ and $(\Pi'(t))_{t\geq0}$ defined on probability spaces with probability measures $P^*$ and $P'$ and expectation operators $E^*$ and $E'$ such that, for any $t\geq0$ and any nonnegative measurable function $F$ on $\D_t$,
\begin{equation}\label{loietoile}
E^*[F(\Pi^*(s))_{s\leq t}]=E\Big[\sum_i |(R\Pi(t))_i|^{p^*}F((H^t_i\Pi(s))_{s\leq t})\Big]
\end{equation}
and
	\[E'[F(\Pi'(s))_{s\leq t}]=E\big[|\Pi_1(t)|^{p^*-1}F((\Pi(s))_{s\leq t})\big].
\]

Just as in the previous section, these two definitions are in fact equivalent:
\begin{prop}\label{aaa} The two processes $(\Pi^*(t))_{t\geq0}$ and $(\Pi'(t))_{t\geq0}$ have the same law.
\end{prop}

To prove this, we only need to show that these two processes have the same finite-dimensional marginal distributions. The $1$-dimensional marginals have already been proven to be the same and we will continue with an induction argument which uses the fact that the homogeneous fragmentation property generalizes to $P^*$ and $P'$.

\begin{lemma} Let $t\geq 0$, and $\Psi^*$ and $\Psi'$ be independent copies of respectively $\Pi^*$ and $\Pi'$. Then, conditionally on $(\Pi^*(s),s\leq t)$, the process $(\Pi^*(t+s))_{s\geq0}$ has the same law as $(\Pi^*(t)\cap\Psi^*(s))_{s\geq0}$ and, conditionally on $(\Pi'(s),s\leq t)$, the process $(\Pi'(t+s))_{s\geq0}$ has the same law as $(\Pi'(t)\cap\Psi'(s))_{s\geq0}.$
\end{lemma}

\begin{proof}
Let $t\geq 0$ and $u\geq 0$, let $F$ be a nonnegative measurable function on $\D_t$ and $G$ be a nonnegative measurable function on $\D_u$. We have, by the fragmentation property,
\begin{align*}
E^*[F((\Pi^*(s))&_{0\leq s\leq t}) G((\Pi^*(t+s))_{0\leq s\leq u})] \\
 &= E\big[\sum_i |(R\Pi(t+u))_i|^{p^*}F((H^{t+u}_iR\Pi(s))_{0\leq s\leq t}) G((H^{t+u}_iR\Pi(t+s))_{0\leq s\leq u}) \big] \\
        &= E\big[\sum_i |R (\Pi(t)\cap\Psi(u))_i|^{p^*} F((H^{t+u}_iR\Psi)_{0\leq s\leq t}) G(H^{t+u}_i(R\Pi(t)\cap\Psi(s))_{0\leq s\leq u}\big],
\end{align*}
where $\Psi$ is an independent copy of $\Pi$.
The key now is to notice that a block of $\Pi(t)\cap\Psi(s)$ is the intersection of a block of $\Pi(t)$ and a block of $\Psi(s)$. Thus we replace our sum over integers $i$ (representing blocks of $\Pi(t)\cap\Psi(s)$) by two sums, one for the blocks of $\Pi(t)$ and another for those of $\Psi(s)$.

\begin{align*}
E^*&[F((\Pi(s))_{0\leq s\leq t}) G((\Pi(t+s))_{0\leq s\leq u})] \\&= E\big[\sum_i \sum_j |R\Pi_i(t)|^{p^*} |R\Psi_j(t)^{p^*}| F((H^{t}_iR\Pi(s))_{0\leq s\leq t}) G((H^t_i R\Pi(t)\cap H^u_j\Psi(s))_{0\leq s\leq u})\big] \\
                                                                   &= E^*[F((\Pi(s))_{0\leq s\leq t}) G((\Pi(t)\cap\Psi(s))_{0\leq s\leq u})].
\end{align*}

The proof for $\Pi'$ again uses the same ideas but is simpler, so we will omit it.
\end{proof}

We can now complete the proof of Proposition \ref{aaa}. Let $t_1<t_2<\ldots<t_{n+1}$ and assume that we have shown that $(\Pi^*(t_1),\ldots,\Pi^*(t_n))$ and $(\Pi'(t_1),\ldots,\Pi'(t_n))$ have the same law. Let $\Psi$ be an independent copy of $\Pi^*(t_{n+1}-t_n)$ (which is then also an independent copy of $\Pi^*(t_{n+1}-t_n)$, then

\begin{align*} (\Pi^*(t_1),\ldots,\Pi^*(t_{n+1})) &\overset{(d)}= (\Pi^*(t_1),\ldots,\Pi^*(t_n),\Pi^*(t_n)\cap\Psi) \\ 
&\overset{(d)}=(\Pi^*(t_1),\ldots,\Pi^*(t_n),\Pi^*(t_n)\cap\Psi) \\
&\overset{(d)}=(\Pi'(t_1),\ldots,\Pi'(t_{n+1})),
\end{align*}
and the proof is complete.
\qed

\medskip

We can now proceed to the main part of this section, which is the description of $\Pi^*$ with Poisson point processes. First, we let $k_{\nu}^*$ be the measure on $\p_{\N}$ defined by 
	\[k_{\nu}^*(d\pi)=|\pi_1|^{p^*-1}\mathbf{1}_{\{|\pi_1|\neq0\}}k_{\nu}(d\pi).
\]

Let $\Delta^1(t)_{t\geq0}$ be a P.p.p. with intensity $k_{\nu}^*$ and, for all $k\geq2$, $(\Delta^k(t))_{t\geq0}$ a P.p.p. with intensity $k_{\nu}$. Let also $T_2,T_3,\ldots$ be exponential variables with parameter $c$ (note that there is no $T_1$ in here). We assume that these variables are all independent. With these, we can create a $\p_{\N}$-valued process $\Pi^*$, just as is done in the case of classical fragmentation processes. We start with $\Pi^*(0)=(\N,\emptyset,\emptyset,\ldots)$. For every $t$ such that there is an atom $\Delta^k(t)$, we let $\Pi^*(t)$ be equal to $\Pi^*(t^-)$, except that we replace the block $\Pi_{k}^*(t^-)$ by its intersection with all the blocks of $\Delta^k(t)$. Also, for every $i$, we let $\Pi^*(T_i)$ be equal to $\Pi^*(T_i^-)$, except that the integer $i$ is removed from its block and placed into a singleton. Just as in the classical case, it might not be clear that this is well-defined. To make sure that it is the case, we are going to restrict this to finite subsets of $\N$. Let $n\in\N$, we now only need to look at integers $k\leq n$ and times $t$ such that $\Delta^k(t)$ splits $[n]$ into at least two blocks. Conveniently enough, this set is in fact finite: indeed, we have 
	\[k_{\nu}(\{[n] \text{ is split into two or more blocks}\})=\int_{\s}  (1-\sum_{i=1}^\infty s_i^n)d\nu(\mathbf{s})\leq \int_{\s} (1-s_1^n)d\nu(\mathbf{s})<\infty,
\]
as well as 
\begin{align*}
k_{\nu^*}(\{[n] \text{ is split into two or more blocks}\})&=\int_{\s}  (1-\sum_{i=1}^\infty s_i^n)\sum_{i=1}^\infty s_i^{p^*}d\nu(\mathbf{s}) \\
                            &=cp^*+ \int_{\s} (1-\sum_{i=1}^\infty s_i^n\sum_{i=1}^\infty s_i^{p^*})d\nu(\mathbf{s}) \\
                            &\leq \int_{\s} (1-s_1^{p^*+n})d\nu(\mathbf{s}) \\
                            &<\infty.
\end{align*}
Since the set $(T_2,\ldots,T_n)$ is also finite, the previous operations can be applied without ambiguity. From this, we get, for all $t$, a sequence $(\Pi^*(t)\cap{[n]})_{n\in\N}$ of compatible partitions, which determine a unique partition $\Pi^*(t)$ of $\N$.

\begin{theo} The process $(\Pi^*(t))_{t\geq0}$ constructed does have the law defined by \ref{loietoile}.
\end{theo}

\begin{proof}
We start by extending the measure $P'$, so that it contains not only the fragmentation process, but also the underlying Poisson point processes and exponential variables: for $t\leq0$, and any nonegative measurable function $F$, let
	\[E'_t\big[F\big((\Delta^i(t)_{s\leq t})_{i\in\N},(T_i)_{i\in\N'}\big)\big]=E\big[|\Pi_1(t)|^{p^*-1}F\big((\Delta^i(t)_{s\leq t})_{i\in\N},(T_i)_{i\in\N'}\big)\big]
\]
(remember that, under $P$, $(\Delta^1(t)_{s\leq t})$ is a P.p.p. with intensity $k_\nu$, and not $k_\nu^*$.)
These probability measures are still compatible, and we can still use Kolmogorov's theorem to extend them to a single measure $P'$. Note that under $P'$, $1$ never falls in a singleton, which is why we have ignored $T_1$. With this new law $P'$, the partition-valued process $(\Pi'(t))_{t\geq0}$ is indeed built from the point processes $(\Delta^k(t))_{t\geq0}$ with $k\in \N$ and the $T_i$ with $i\in \N$, and all we need to do is now find their joint distribution. We start with the harder part, which is finding the law of $(\Delta^1(t))_{t\geq0}$, and will use a Laplace transform method and the exponential formula for Poisson point processes. If $t\geq0$ and $f$ is a nonnegative measurable function on $\p_{\N}\times \R$, we have

\begin{align*}
E'[\e^{-\sum_{s\leq t} f(\Delta^1_s,s)}] &= E\big[|\Pi_{1}(t)|^{p^*-1} \mathbf{1}_{\{|\Pi_{1}(t)|\neq 0\}}e^{-\sum_{s\leq t} f(\Delta^1_s,s)}\big] \\
                                                        &= \e^{-ct}\e^{-ct(p^*-1)}E\Big[\prod_{s\leq t} |\Delta^1_1(s)|^{p^*-1} \mathbf{1}_{|\Delta^1_{1}(s)|\neq 0}e^{-\sum_{s\leq t} f(\Delta^1(s),s)}\Big]\\
                                                        &= \e^{-ctp^*}E\left[\exp\left(-\sum_{s\leq t} (-(p^*-1)\log (|\Delta^1_1(s)|) + f(\Delta^1(s),s))\right)\right] \\
                                                        &= \e^{-ctp^*}\exp\left(-\int_0^t\int_{\p_{\N}} (1-\e^{-(-(p^*-1)\log (|\pi_1|)+f(\pi,s))})k_{\nu}(d\pi)ds\right) \\
                                                        &= \e^{-ctp^*}\exp\left(-\int_0^t\int_{\p_{\N}} (1-|\pi_1|^{p^*-1}e^{-f(\pi,s)})k_{\nu}(d\pi)ds\right)
\end{align*}
Now we use the the Malthusian hypothesis: we have $cp^*+\int_{\s} (1-\sum s_i^{p^*})d\nu(\mathbf{s})=0.$ Translating this in terms of $k_{\nu}$, we have 
\begin{align*}
\int_{\p_{\N}}(1-|\pi_1|^{p^*-1})k_{\nu}(d\pi)&=\int_{\s}\big(\sum_i s_i(1-s_i^{p^*-1}) +s_0\big)d\nu(\mathbf{s}) \\
                                              &= -cp^*.
\end{align*}
Thus, in the last integral with respect to $k_{\nu}$, we can replace $1$ by $|\pi_1|^{p^*-1}$, if we subtract $cp^*$ outside of the integral:
\begin{align*}
E'[\e^{-\sum_{s\leq t} f(\Delta^1_s,s)}] &=\e^{-ctp^*}\exp\left(-\int_0^t(-cp^*+\int_{\p_{\N}} (|\pi_1|^{p^*-1}-|\pi_1|^{p^*-1}e^{-f(\pi,s)})k_{\nu}(d\pi))ds\right) \\
                                         &=\exp\left(-\int_0^t\int_{\p_{\N}} (1-e^{-f(\pi,s)})|\pi_1|^{p^*-1}k_{\nu}(d\pi)ds\right).
\end{align*}
This means that the point process $(\Delta^1(t))_{t\geq0}$ does indeed have the law of a Poisson point process with intensity $|\pi_1|^{p^*-1}k_{\nu}(d\pi)$.

Let us now prove that the point processes and random variables are independent from each other and that, except for $(\Delta^1_t)_{t\geq0}$, they have the same law as under $P$. Take $n\in\N$ and $t\geq 0$, for every $i\in[n]$, $F_i$ a nonnegative measurable function on the space of random measures on $\p_{\N}\times[0,t]$, and for $2\leq i\leq n$, a nonnegative measurable function $g_i$ on $\R$. Using independence properties under $P$, we have

\begin{align*}
E'&\left[\prod_{i=1}^n F_i((\Delta^i(s))_{s\leq t})\prod_{i=2}^{n}g_i(T_i)\right]\\&=E\left[\prod_{s\leq t} |\Delta^1_1(s)|^{p^*-1} \mathbf{1}_{|\Delta^1_{1}(s)|\neq 0} F_1((\Delta^1(s))_{s\leq t})\prod_{i=2}^{n}F_i((\Delta^i(s))_{s\leq t})g_i(T_i)\right] \\
   &= E\left[\prod_{s\leq t} |\Delta^1_1(s)|^{p^*-1} \mathbf{1}_{|\Delta^1_{1}(s)|\neq 0} F_1((\Delta^1(s))_{s\leq t})\right] \prod_{i=2}^{n} E\big[F_i((\Delta^i(s))_{s\leq t})\big]\prod_{i=2}^{n}E[g_i(T_i)] \\
   &=E'[F_1\big((\Delta^1(s)_{s\leq t}))\big] \prod_{i=2}^{n} E\big[F_i((\Delta^i(s)_{s\leq t}))\big]\prod_{i=2}^{n}E[g_i(T_i)],
\end{align*}

which is all we need.
\end{proof}

\begin{rem}\label{31} Here is an alternative description of a Poisson point process $(\Delta^1(t))_{t\geq0}$ with intensity $k_{\nu}^*$. Let $(s(t),i(t))_{t\geq0}$ be a $\s\times\N$-valued Poisson point process with intensity $s_i^{p^*}d\nu(\mathbf{s})d\#(i)$, where $\#$ is the counting measure on $\N$ (otherwise said, $(s(t))_{t\geq0}$ has intensity $\sum_i s_i^{p^*} d\nu(\mathbf{s})$ and $i(t)$ is equal to an integer $j$ with probability $\frac{s_j^{p^*}}{\sum_i s_i^{p^*}}$). When there is an atom, construct a partition of $\N'$ using the paintbox method (using for example a coupled process of uniform variables), and then add $1$ to the $i(t)$-th block, where the blocks are ordered in decreasing order of their asymptotic frequencies.
\end{rem}

\subsection{Link between $\mu^*$ and $P^*$.}
Let $\T$ be the fragmentation tree derived from $\Pi^{\alpha}$, equipped with its list of death points $(Q_i)_{i\in\N}$, as well as the measure $\mu^*$ which has total mass $W$, and we keep the assumption that $E[W]=1$. Given any leaf $L$, we can build a new partition process $(\Pi_L^{\alpha}(t))_{t\geq0}$ from this, by declaring the "new death point" of $1$ to be $L$. More precisely, for all $t\geq0$, the restriction of $\Pi_L^{\alpha}(t)$ to $\N'$ is the same as that of $\Pi^{\alpha}(t)$, while $1$ is put in the block containing all the integers $j$ such that $Q_j$ is in the same tree component of $\T_{>t}$ as $L$. As in the proof of Proposition \ref{1more}, one can show that $\Pi_L^{\alpha}$ is decreasing and in $\D$. Our main result here is that, if $L$ is chosen with "distribution" $\mu^*$, then $\Pi^{\alpha}_L$ has the same distribution as the $\Pi^{*,\alpha}$, where $\Pi^{*,\alpha}$ is the "$\alpha$-self-similar" version of $\Pi^*$, obtained through the usual time-change.

\begin{prop}\label{32} Let $F$ be any nonnegative measurable function of $\D$, then $\int_{\T}F(\Pi^{\alpha}_L)d\mu^*(L)$ is a random variable and we have
	\[E\left[\int_{\T}F(\Pi^{\alpha}_L)d\mu^*(L)\right]=E^*[F(\Pi^{*,\alpha})].
\]
\end{prop}

\begin{proof} For any leaf $L$ of $\T$, we let $\Pi_L=G^{-\alpha}(\Pi^{\alpha}_L)$, then $\Pi^{\alpha}_L=G^{\alpha}(\Pi_L)$ (recall from Section $2.1.3$ that $G^{\alpha}$ and $G^{-\alpha}$ are the measurable functions which transform $\Pi$ to $\Pi^{\alpha}$ and back). By renaming, we are reduced to proving that, for any nonnegative measurable function $F$ on $\D$, $\int_{\T}F(\Pi_L)d\mu^*(L)$ is a random variable and
	\[	E\left[\int_{\T}F(\Pi_L)d\mu^*(L)\right]=E^*[F(\Pi^{*})].
\]
We let $M(F)=\int_{\T}F(\Pi_L)d\mu^*(L)$. Assume first that $F$ is of the form $F\big((\pi(s))_{s\geq0}\big)=K\big((\pi(s))_{0\leq s\leq t}\big)$, for a certain $t\geq0$ and a function $K$ on $\D_t.$ We then have, by definition of $\mu^*$,
	\[M(F)=\sum_i |R\Pi_i(t)|^{p^*}X_i K((H_i^t(R\Pi)(s))_{0\leq s \leq t}),
\]
where $X_i$ is defined for all $i$ by $X_i=\frac{W_{j,t}}{|R\Pi_i(t)|^{p^*}}$ for any choice of $j\in\Pi_i(t)$, so $X_i$ has the same law as $W$ and is independent of $(\Pi(s))_{s\leq t}$. We thus know that $M(F)$ is a random variable such that
	\[E[M(F)]=E[W]E\big[\sum_i |R\Pi_i(t)|^{p^*} K((H_i^t(R\Pi)(s))_{0\leq s \leq t})\big]=E^*[F(\Pi^*)].
\]
A measure theory argument then extends this to any nonnegative measurable function $F$. Let $\mathcal{A}$ be the set of measurable subsets $A\in\D$ such that $M(\mathbf{1}_A)$ is a random variable and $E[M(\mathbf{1}_A)]=P^*[\Pi^*\in A]$. Standard properties of integrals show that $\mathcal{A}$ is a monotone class, and since it contains the generating $\pi$-system of sets of the form $A=\{\pi\in\D,\:(\pi(s))_{0\leq s\leq t}\in B\}$ with $t\geq0$ and $B\subset \D_t$, the monotone class theorem implies that $\mathcal{A}$ is $\D$'s Borel $\sigma$-field. We then conclude by approximating $F$ by linear combinations of indicator functions.


\end{proof}

\subsection{Marking two points}
We now want to go further and mark two points on $\T$ with distribution $\mu^*$. However, in order to avoid having to manipulate partitions with both integers $1$ and $2$ being forced into certain blocks, we will instead work with the tree $\T^*=\TREE(\Pi^{*,\alpha})$. To make sure that this is properly defined, we need to check that $\Pi^{*,\alpha}$ satisfies the hypotheses of Lemmas \ref{13} and \ref{14}. The first one is immediate because, for all $t\geq0$, when restricted to the complement of $\Pi^{*,\alpha}_1(t)$, $(\Pi^{*,\alpha}(s)_{s\geq t})$ is an $\alpha$-self-similar fragmentation process, while the second one comes from the Poissonian construction.

Let us give an alternate description of $\T^*$ which we will use here. Let $(\Delta(t))_{t\geq0}$ be a Poisson point process with intensity measure $\kappa_{\nu}^{*}$, and, for all $t\geq0$, $\xi(t)=\e^{-ct}\prod_{s\leq t} |\Delta(s)|$. From this we define the usual time-change: for all $t\geq0, \tau(t)=\inf \{u, \int_0^u \xi(t)^{-\alpha}dr >t\}.$ The tree $\T^*$ is then made of a \emph{spine} of length $T=\tau^{-1}(\infty)$ on which we have attached many small independent copies of $\T$. More precisely, for each $t$ such that $(\Delta(s))_{s\geq0}$ has an atom at time $\tau(t)$, we graft on the spine at height $t$ a number of trees equal to the number of blocks of $\Delta(t)$ minus one (an infinite amount if $\Delta_t$ has infinitely many). These are indexed by $j\geq 2$ and, for every such $j$, we graft precisely a copy of $\Big((\xi(t^-)|\Delta_j(t)|)^{-\alpha}\T,(\xi(t^-)|\Delta_j(t)|)\mu\Big)$, which will be called $(\T'_{j,t}, \mu'_{j,t})$. All of these then naturally come with their copy of $\mu^*$ which we will call $\mu^*_{i,t}$. These can then all be added to obtain a measure $\mu^{**}$ on $\T$, which satisfies, for all $(i,t) \in \T^*$,
	\[\mu^{**}(\T^*_{i,t+})=\underset{s\to\infty}\lim \sum_{j\in\Pi^*(\tau_i(t)+s)\cap \rep(\Pi^*(\tau_i(t)))} |\Pi^*_j(t+s)|^{p^*}.
\]
The measure $\mu^{**}$ is the natural analogue of $\mu^*$ on the biased tree.

We will need a Gromov-Hausdorff-type metric for \emph{trees with two extra marked points}: let $(\T,\rho,d)$ and $(\T',\rho',d')$ be two compact rooted trees, and then let $(x,y)\in\T^2$ and $(x',y')\in(\T')^2$. We now let the $2$-pointed Gromov-Hausdorff $d^2_{GH}((\T,x,y),(\T',x',y'))$ be equal to
	\[\inf \big[ \max \big(d_{\mathcal{Z},H} (\phi(\T),\phi'(\T')),d_\mathcal{Z}(\phi(\rho),\phi'(\rho')),d_\mathcal{Z}(\phi(x),\phi'(x')),d_\mathcal{Z}(\phi(y),\phi'(y'))\big)\big],
\]
where the infimum is once again taken on all possible isometric embeddings $\phi$ and $\phi'$ of $\T$ and $\T'$ in a common space $\mathcal{Z}$. Taking classes of such trees up to the relation $d^2_{GH}$, we then get a Polish space $\TT^2$ which is the set of $2$-pointed compact trees. For more details in a more general context (pointed metric spaces instead of trees), the reader can refer to \cite{M09}, Section $6.4$.

\begin{prop}\label{34} Let $F$ be any nonnegative measurable function on $\TT^2$. Then $\\ \int_{\T} F(\T,L,L')d\mu^*(L)d\mu^*(L')$ is a random variable, and we have
	\[E\left[\int_{\T}\int_{\T} F(\T,L,L')d\mu^*(L)d\mu^*(L')\right]=E^*\left[\int_{\T^*}F(\T,L_1,L')d\mu^{**}(L')\right].
\]
\end{prop}

\begin{proof} As in the proof of Proposition \ref{32}, we let $\Pi_L^{\alpha}$ be the fragmentation-like process obtained by setting the leaf $L$ as the new death point of the integer $1$ in $\T$, and then we let $\Pi_L$ be its homogeneous version. The other leaf $L'$ will be represented by a sequence of integers $(j_{L'}^{\alpha}(t))_{0\leq t< ht(L')}$ where, for all $t$ with $0\leq t < ht(L')$, $j^{\alpha}_{L'}(t)$ is the smallest integer $j\neq1$ such that $(j,t)\leq L'$ in $\T^*$. We then let $(j_{L'}(t))_{t\geq0}$ we the image of $(j^{\alpha}_{L'}(t))_{0\leq t\leq ht(L')}$ through the reverse Lamperti transformation.

Notice that $(\T,L,L')$ is the image of $(\Pi_L(t),j_{L'}(t))_{t\geq0}$ by a measurable function. Indeed, going back to the representation in $\ell^1$ of our trees, $\T$ is no more than $\TREE(\Pi_L^{\alpha})$, $L_1$ is $Q_1$, while $L'$ is the limit as $t$ goes to infinity of $Q_{j_{L'}(t)}$.

Thus, with some renaming, we now just need to check that, if $F$ is a nonnegative measurable function on the space of $\p_{\N}\times \N$-valued càdlàg functions (equipped with the product $\sigma$-algebra generated by the evaluation functions), then $\int_{\T}F((\Pi_L(t),j_{L'}(t))_{t\geq0})d\mu^*(L)d\mu^*(L')$ is a random variable, and 
	\[E\left[\int_{\T}\int_{\T} F((\Pi_L(t),j_{L'}(t))_{t\geq0})d\mu^*(L)d\mu^*(L')\right]=E^*\left[\int_{\T^*}F((\Pi^*(t),j_{L'}(t))_{t\geq 0})d\mu^{**}_{i,t}(L')\right].
\]
This will be done the same way as before: suppose that $F$ is of the form $K((\pi(s),j(s))_{0\leq s \leq t})$, then one can write
	\[\int_{\T}\int_{\T} F((\Pi_L(t),j_{L'}(t))_{t\geq0})d\mu^*(L)d\mu^*(L')= 
                              \int_{\T}\sum_j W_{j(t),t} K((\Pi_L(s),j(s))_{0\leq s\leq t})d\mu^*(L).
                              \]
(In the right-hand side, $j(s)$ denotes the smallest element of the block of $\Pi_L(s)$ which contains $(\Pi_L(t))_j$.) By Proposition \ref{32}, this is a random variable, and we know that its expectation is equal to 
	\[E^*\left[\sum_j |\Pi^*_j(t)|^{p^*} K((\Pi^*(s),j(s))_{0\leq s\leq t})\right] = E^*\left[\int_{t^*}F((\Pi^*(t),j_{L'}(t))_{t\geq0})d\mu^{**}(L')\right].
\]
A monotone class argument similar to the one at the end of Proposition \ref{32} ends the proof.

\end{proof}
\section{The Hausdorff dimension of $\T$}
The reader is invited to read \cite{Falconer} for the basics on the Hausdorff dimension $\dim_{\mathcal{H}}$ of a set, which we will not recall here.
\subsection{The result}
\begin{theo}\label{theo} Assume $\mathbf{(H)}$, that is that the function $\psi$ takes at least one strictly negative value on $[0,1]$. Then there exists a Malthusian exponent $p^*$ for $(c,\nu)$ and, almost surely, on the event that $\Pi$ does not die in finite time, we have
	\[\dim_{\mathcal{H}}(\mathcal{L}(\T))= \frac{p^*}{|\alpha|}.
\]
If $\Pi$ does die in finite time, then the leaves of $\T$ form a countable set, which has dimension $0$.
\end{theo}

The last statement is a consequence of Proposition \ref{leaf}: if $\Pi$ does die in finite time, then there are no proper leaves, which implies that every leaf of $\T$ is the death point of some integer.
\subsection{The lower bound}
An elaborate use of Frostman's lemma (Theorem 4.13 in \cite{Falconer}) with the measure $\mu^*$ combined with a truncation of the tree similar to what was done in \cite{HM04} will show that $\dim_{\mathcal{H}}(\mathcal{L}(\T))\geq \frac{p^*}{|\alpha|}$ almost surely when $\Pi$ does not die in finite time.

\subsubsection{A first lower bound}
Here we assume that $E[W]=1$, and thus $\Pi$ dies in finite time if and only if $\mu^*$ is the zero measure. We also assume the integrability condition $\int_{\s}(\sum_i |log(s_i)|s_i^{p^*}) d\nu(\mathbf{s})<\infty$ of Lemma \ref{25}.
\begin{lemma}\label{41} Recall that $\underline{p}=\sup \{q\in \R: \phi^*(-q)>-\infty\}$. Let 
	\[A= \sup \{a\leq p^*: \int_{\s} \sum_{i\neq j} s_i^{p^*-a}s_j^{p^*}d\nu(\mathbf{s}) < \infty,\} \in [0,p^*].
\]
Then, on the even where $\Pi$ does not die in finite time, we have the lower bound:
	\[\dim_{\mathcal{H}}(\mathcal{L}(\T))\geq \frac{A\wedge(|\alpha|+\underline{p})}{|\alpha|}.
\]
\end{lemma}
\begin{proof}
We want to apply Proposition \ref{34} to the function $F$ defined on $\TT^2$ by $F(\T,\rho,d,x,y)=d(x,y)^{-\gamma}\mathbf{1}_{x\neq y}$. To do this we need to check that it is measurable, which can be done by showing that $d(x,y)$ is continuous. In fact, it is even Lipschitz-continuous: for all $(\T,\rho,d,x,y)$ and $(\T',\rho',d',x',y')$ and any embeddings $\phi$ and $\phi'$ of $\T$ and $\T'$ in a common $\mathcal{Z}$, we have
	\[|d(x,y)-d'(x',y')|=|d_\mathcal{Z}(\phi(x),\phi(y))-d_\mathcal{Z}(\phi'(x'),\phi'(y'))|\leq d_\mathcal{Z}(\phi(x),\phi'(x'))+d_\mathcal{Z}(\phi(y),\phi'(y'))
\]
and then taking the infimum, we obtain
	\[|d(x,y)-d'(x',y')|\leq 2d^2_{GH}\big((\T,x,y),(\T',x',y')\big).
\]

\smallskip

Applying Proposition \ref{34} to $F$, we then get
	\[E\left[\int_{\T}\int_{\T} (d(L,L'))^{-\gamma} d\mu^*(L)d\mu^*(L')\right]=E^*\left[\int_{\T^*}(d(L_1,L'))^{-\gamma}d\mu^{**}(L')\right].
\]

Recall the Poisson description of $\T^*$ of Section 5.4. Let, for all relevant $j\geq 2$ and $t\geq0$, $X_{j,t}$ be the root of $T'_{j,t}$ and $Z_{j,t}=\int_{\T'_{j,t}} d(L',X_{k,t})^{-\gamma}d\mu^*(L')$ One can then write $Z_{j,t}=\big(\xi(t^-)|\Delta_j(t)|\big)^{p^*+\alpha\gamma} (I_{j,t})^{-\gamma}$ where $I_{i,t}$ is a copy of $I$ (defined in the proof of Lemma \ref{25}) which is independent from the process $(\Delta)_{t\geq0}$ and all the other $\T'_{k,s}$ for $(k,s)\neq(j,t)$. Thus, the process $(\Delta_t,(I_{j,t})_{j\geq2})_{t\geq0}$ is a Poisson point process whose intensity is the product of $\kappa_{\nu}^*$ and the law of an infinite sequence of i.i.d. copies of $I$. We then have

\begin{align*}
 E^*\big[\int d(L_1,L')^{\gamma}d\mu^*(L')\big] &= E^*\big[\sum_{t\geq0}\sum_{j\geq2}\int_{\T'_{j,t}}d(L_1,L')^{-\gamma}d\mu^{**}(L')\big] \\
                                         &\leq E^*\big[\sum_{t\geq0}\sum_{j\geq2}\int_{\T'_{j,t}}d(L',X_{i,t})^{-\gamma}d\mu^{**}(L')\big] \\
                                         &=E^*\big[\sum_{t\geq0}\sum_{j\geq2} \big(\xi(t^-)|\Delta_j(t)|\big)^{p^*+\alpha\gamma} (I_{j,t})^{-\gamma}\big] \\
                                            &= E\big[I^{-\gamma}]E^*[\int \xi_{t^-}^{p^*+\alpha\gamma} dt\big] \int_{\s}\sum_i s_i^{p^*} \sum_{j\neq i}s_j^{p^*+\alpha\gamma}d\nu(\mathbf{s}).
\end{align*}
The last equality directly comes from the Master Formula for Poisson point processes.

We have a product of three factors, and we want to know when they are finite. The case of the first factor has already been studied in Lemma \ref{25}, we know that it is finite when $\gamma<1+\frac{\underline{p}}{|\alpha|}$. For the second factor to be finite we simply need $\phi^*(p^*+\alpha\gamma)>0$, which is true as soon as $p^*+\alpha\gamma>0$ i.e. when $\gamma < \frac{p^*}{|\alpha|}$. Finally, by definition of $A$, the third factor is finite as soon as $\gamma < \frac{A}{|\alpha|}.$ Since $A\leq p^*$ by definition, Frostman's lemma implies Lemma \ref{41}.
\end{proof}

\subsubsection{A reduced fragmentation and the corresponding subtree}
Let $N\in \N$ and $\epsilon>0$, we define a function $G_{N,\epsilon}$ from $\s$ to $\s$ by
	\[G_{N,\epsilon}(\mathbf{s})= 
	\begin{cases}
	(s_1,\ldots,s_N,0,0,\ldots) & \text{if } s_1\leq 1-\epsilon \\
   (s_1,0,0,\ldots) & \text{if } s_1>1-\epsilon.
\end{cases}
\]
A similar function can be defined on partitions on $\p_{\N}$. If a partition $\pi$ does not have asymptotic frequencies (a measurable event which doesn't concern us), we let $G_{N,\epsilon}(\pi)=\pi$. If it does, we first reorder its blocks by decreasing order of their asymptotic frequencies by letting, for all $i$, $\pi_i^{\downarrow}$ be the block with $i$-th highest asymptotic frequency (if there is a tie, we just rank those blocks by increasing order of their first elements). Then we let
	\[G_{N,\epsilon}(\pi)= 
	\begin{cases}
	(\pi_1^{\downarrow},\ldots,\pi_N^{\downarrow},\text{singletons}) & \text{if } |\pi_1^{\downarrow}|\leq 1-\epsilon \\
   (\pi_1^{\downarrow},\text{singletons}) & \text{if } |\pi_1^{\downarrow}|>1-\epsilon.
\end{cases}
\]

We let $\nu_{N,\epsilon}$ be the image of $\nu$ by $G_{N,\epsilon}$. Then the image of $k_{\nu}$ by $G_{N,\epsilon}$ on $\p_{\N}$ is $k_{\nu_{N,\epsilon}}$. The following is immediate.

\begin{prop} Let $(\Delta_t,k_t)_{t\geq0}$ be a Poisson point process with intensity $k_{\nu}\otimes\#$, then $(G_{N,\epsilon}(\Delta_t),k_t)_{t\geq0}$ is a Poisson point process with intensity $k_{\nu_{N,\epsilon}}\otimes \#$. Using them, one gets two coupled fragmentation processes $(\Pi(t))_{t\geq0}$ and $(\Pi^{N,\epsilon}(t))_{t\geq0}$ such that, for all $t$, $\Pi^{N,\epsilon}(t)$ is finer than $\Pi(t)$. Also, $\T_{N,\epsilon}$, the tree built from $(\Pi^{N,\epsilon}(t))_{t\geq0}$, is naturally a subset of $\T$.
\end{prop}

\subsubsection{Using the reduced fragmentation}

Recall the concave function $\psi$ defined from $\R$ to $[-\infty,+\infty)$ by 
	\[\psi(p)=cp+\int_{\s} (1-\sum_i s_i^p) d\nu(\mathbf{s}).
\]
We now assume $\mathbf{(H)}$: there exists $p>0$ such that $-\infty<\psi(p)<0$.

\begin{prop} For $N\in \N\cup\{\infty\}$, $\epsilon\in[0,1]$ and $p\in\R$, let $ \psi_{N,\epsilon}(p)=cp+\int_{\s} (1-\sum_i s_i^p) d\nu_{N,\epsilon}(\mathbf{s}).$ One can then write
	\[\psi_{N,\epsilon}(p) = cp+\int_{\s}\left( \Big(1-\sum_{i=1}^ N s_i^p\Big)\mathbf{1}_{\{s_1\leq 1-\epsilon\}} + (1-s_1^p)\mathbf{1}_{\{s_1> 1-\epsilon\}} \;\right)d\nu(\mathbf{s}).
\]

  (i)  This is a nonincreasing function of $N$ and a nondecreasing function of $\epsilon$.
  
  (ii) We have $\psi(p)=\underset{N,\epsilon} \inf \, \psi_{N,\epsilon}(p)$. 
  
  (iii) There exist $N_0$ and $\epsilon_0$ such that, for $N>N_0$ and $\epsilon<\epsilon_0$, the pair $(c,\nu_{N,\epsilon})$ satisfies $\mathbf{(H)}$ and has a Malthusian exponent $p^*_{N,\epsilon}$.
  
  (iv) We have $p^*=\underset{N,\epsilon} \sup \, p^*_{N,\epsilon}$.
\end{prop}

\begin{proof} The first point is immediate. The second one is a straightforward application of the monotone convergence theorem as $N$ tends to infinity and $\epsilon$ tends to $0$, which is valid because we have, for all $\mathbf{s}$, the upper bound \[\left(1-\sum_{i=1}^N s_i^p\right)\mathbf{1}_{\{s_1\leq 1-\epsilon\}} + (1-s_1^p)\mathbf{1}_{\{s_1> 1-\epsilon\}} \leq (1-s_1^p)\leq C_p(1-s_1),
\] and $(1-s_1)$ is $\nu$-integrable. 

The third point is a direct consequence of the second: let $p\in[0,1]$ such that $\psi(p)<0$, there exist $N_0$ and $\epsilon_0$ such that $\psi_{N_0,\epsilon_0}(p)<0$. Then by monotonicity, for all $N> N_0$ and $\epsilon<\epsilon_0$, $\psi_{N,\epsilon}(p)<0$ and thus $\nu_{N,\epsilon}$ has a Malthusian exponent $p_{N,\epsilon}^*$.

Now for the last point: first notice that, for all $N$ and $\epsilon$, we have $\phi_{N,\epsilon}(p^*)\geq \phi(p^*)=0$ and thus, if it exists, $p^*_{N,\epsilon}$ is smaller than or equal to $p^*$. Then, for $p<p^*$, by taking $N$ large enough and $\epsilon$ small enough, we have $\psi_{N,\epsilon}(p)<0$ and thus $p_{N,\epsilon}^*\geq p$. This concludes the proof.

\end{proof}

\begin{prop} For all $N$ and $\epsilon$ such that $p^*_{N,\epsilon}$ exists, and for all $q>1$, the measure $\nu_{N,\epsilon}$ satisfies assumption $\mathbf{(M_q)}$: $\int_{\s}|1-\sum_{i=1}^{\infty}s_i^{p^*_{N,\epsilon}}|^q\,d\nu_{N,\epsilon}(\mathbf{s}) <\infty$.
\end{prop}

\begin{proof} It is simply a matter of bounding $(1-\sum_{i=1}^ N s_i^{p^*_{N,\epsilon}})\mathbf{1}_{s_1\leq 1-\epsilon} + (1-s_1^{p^*_{N,\epsilon}})\mathbf{1}_{s_1> 1-\epsilon}$ in such a way that both the upper and lower bound's absolute values have an integrable $q$-th power. For the upper bound, write
\[\left(1-\sum_{i=1}^N s_i^{p^*_{N,\epsilon}}\right)\mathbf{1}_{\{s_1\leq 1-\epsilon\}} + (1-s_1^{p^*_{N,\epsilon}})\mathbf{1}_{\{s_1> 1-\epsilon\}} \leq 1-s_1^{p^*_{N,\epsilon}} \leq C_{p^*_{N,\epsilon}}(1-s_1)
\]
and since $q>1$, we can bound $(1-s_1)^q$ by $1-s_1$ which is integrable. For the lower bound, write
	\[ \left(1-\sum_{i=1}^ Ns_i^{p^*_{N,\epsilon}}\right)\mathbf{1}_{\{s_1\leq 1-\epsilon\}} + (1-s_1^{p^*_{N,\epsilon}})\mathbf{1}_{\{s_1> 1-\epsilon\}} > 
	(1-N)\mathbf{1}_{\{s_1\leq 1-\epsilon\}}
\]
and then note that, since $\nu$ integrates $1-s_1$, the set $\{s_1\leq 1-\epsilon\}$ has finite measure.
\end{proof}

\begin{prop} Let $N,\epsilon$ be such that $p^*_{N,\epsilon}$ exists. Let then $A_{N,\epsilon}$ and $\underline{p}_{N,\epsilon}$ corresponding quantities to $A$ and $\underline{p}$ (see Lemma \ref{41}), replacing $\nu$ by $\nu_{N,\epsilon}$. Then $A_{N,\epsilon}=p^*_{N,\epsilon}$ and $\underline{p}_{N,\epsilon}\geq p^*_{N,\epsilon}$.
\end{prop}

\begin{proof}
The important fact to note here is that, since $1-s_1$ is integrable with respect to $\nu$, we have $\nu(\{s_1\leq 1-\epsilon\})<\infty.$ Now notice that, for all $p<p^*_{N,\epsilon}$, we have
\begin{align*}
\int_{\s} \sum_{i=1}^{\infty} (1-s_i^{-p})s_i^{p^*_{N,\epsilon}}&d\nu_{N,\epsilon}(\mathbf{s}) = cp^*_{N,\epsilon}+\int_{\s} \left(1-\sum_{i=1}^{\infty} s_i^{p^*_{N,\epsilon}-p}\right) d\nu_{N,\epsilon}(\mathbf{s}) \\
                           &=cp^*_{N,\epsilon}+ \int_{\s} \left(1-\sum_{i=1}^N s_i^{p^*_{N,\epsilon}-p}\mathbf{1}_{\{s_1\leq 1-\epsilon\}} + (1-s_1^{p^*_{N,\epsilon}-p})\mathbf{1}_{\{s_1> 1-\epsilon\}}\right) d\nu(\mathbf{s}) \\
                           &\geq cp^*_{N,\epsilon}-(N-1)\nu(\{s_1\leq 1-\epsilon\}) \\
                           &> -\infty.
\end{align*}
This shows that $\underline{p}_{N,\epsilon}\geq p^*_{N,\epsilon}$. Similarly, for $a<p^*_{N,\epsilon}$, we have
\begin{align*}
\int_{\s} \sum_{i\neq j} s_i^{p^*_{N,\epsilon}-a}s_j^{p^*_{N,\epsilon}}d\nu_{N,\epsilon}(\mathbf{s}) &=\int_{\s} \sum_{i\neq j \leq N} s_i^{p^*_{N,\epsilon}-a}s_j^{p^*_{N,\epsilon}}\mathbf{1}_{\{s_1\leq 1-\epsilon\}}d\nu(\mathbf{s}) \\
&\leq N^2\nu(\{s_1\leq 1-\epsilon\}) \\
&<\infty.
\end{align*}
Thus $A_{N,\epsilon}=p^*_{N,\epsilon}$
\end{proof}

Combining all the previous results, we have proved the following:

\begin{prop} Assume $\mathbf{(H)}$. Then, on the event where at least one of the $\Pi^{N,\epsilon}$ does not die in finite time, we almost surely have
	\[\dim_{\mathcal {H}}(\T) \geq \frac{\sup_{N,\epsilon} p^*_{N,\epsilon}}{|\alpha|}=\frac{p^*}{|\alpha|}.
\]
\end{prop}

Thus, to complete our proof, we want to check the following lemma:
\begin{lemma} Almost surely, if $\Pi$ does not die in finite time, then for $N$ large enough and $\epsilon$ small enough, $\Pi^{N,\epsilon}$ also does not.
\end{lemma}

\begin{proof} We will argue using Galton-Watson processes. Let, for all integers $n$, $Z(n)$ be the number of non-singleton and nonempty blocks of $\Pi(n)$ and, for all $N$ and $\epsilon$, $Z_{N,\epsilon}(n)$ be the number of non-singleton and nonempty blocks of $\Pi^{N,\epsilon}(n)$. These are Galton-Watson processes, which might take infinite values. We want to show that, on the event that $Z$ doesn't die, there exist $N$ and $\epsilon$ such that $Z_{N,\epsilon}$ also survives. By letting $q$ be the extinction probability of $Z$ and $q_{N,\epsilon}$ be the extinction probability of $Z_{N,\epsilon}$, this will be proved by showing that $q=\underset{N,\epsilon}\inf \, q_{N,\epsilon}$. By monotonicity properties, this infimum is actually equal to $q'=\underset{N\to\infty}\lim q_{N,\frac{1}{N}}$.

Assume that $q<1$ (otherwise there is nothing to prove). This implies that $E[Z(1)]>1$, and by monotone convergence, there exists $N$ such that $E[Z_{N,\frac{1}{N}}(1)]>1$, and thus $q_{N,\frac{1}{N}}<1$. Let, for $x\in[0,1]$, $F(x)=E[x^{Z(1)}]$ and, for all $N$ and $\epsilon$, $F_{N,\epsilon}(x)=E[x^{Z_{N,\epsilon}(1)}]$. The sequence of nondecreasing functions $(F_{N,\frac{1}{N}})_{N\in\N}$ converges simply to $F$. Since $F$ is continuous on the compact interval $[0,q_{N,\frac{1}{N}}]$, the convergence is in fact uniform on this interval. We can take the limit in the relation $F_{N,\frac{1}{N}}(q_{N,\frac{1}{N}})=q_{N,\frac{1}{N}}$ and get $F(q')=q'$. Since $q'<1$ and since $F$ only has two fixed points on $[0,1]$ which are $q$ and $1$, we obtain that $q=q'$.

\end{proof}

We have thus proved the lower bound of Theorem \ref{theo}: assuming $\mathbf{(H)}$, almost surely, if $\Pi$ does not die in finite time, then $\dim_{\mathcal{H}}(\mathcal{L}(\T))\geq \frac{p^*}{|\alpha|}.$

\subsection{Upper bound}
Here we will not need the existence of an exact Malthusian exponent, and we will simply let
	\[p' = \inf \Big\{p\geq 0, \psi(p)\geq 0 \Big\}.
\]

\begin{prop} We have almost surely
	\[\dim_{\mathcal{H}}\big(\mathcal{L(\T)}\big) \leq \frac{p'}{|\alpha|}.
\]

\end{prop}

This statement is in fact slightly stronger than the upper bound of Theorem \ref{theo}. In particular it states that, if there exists $p\leq0$ such that $\psi(p)\geq0$, then the Hausdorff dimension of the set of leaves of $\T$ is almost surely equal to zero.

\begin{proof} We will find a good covering of the set of proper leaves, in the same spirit as in \cite{HM04}, but which takes account of the sudden death of whole fragments. Let $\epsilon>0$. For all $i \in \mathbb{N}$, let 
	\[t_i^{\epsilon} = \inf \{t\geq0: |\Pi_{(i)}(t)|< \epsilon\}.
\] Note that this is in fact a stopping line as defined in section $2.1.4$. We next define an exchangeable partition $\Pi^{\epsilon}$ by saying that integers $i$ and $j$ are in the same block if $\Pi_{(i)}(t_i^{\epsilon})= \Pi_{(j)}(t_j^{\epsilon})$. This should be thought of as the partition formed by the blocks of $\Pi$ the instant they get small enough. Now, for all integers $i$, consider
	\[\tau_{(i)}^{\epsilon}= \underset{j\in \Pi_{(i)}(t_i^{\epsilon})}\sup \inf \{t\geq t_i^{\epsilon}:|\Pi_{(j)}(t)|=0\} - t_i^{\epsilon},
\] the time this block has left before it is completely reduced to dust. This allows us to define our covering. For all integers $i$, we let $b_i^{\epsilon}$ be the vertex of $[0,Q_i]$ at distance $t_i^{\epsilon}$ from the root. We take a closed ball with center $b_i^{\epsilon}$ and radius $\tau_{(i)}^{\epsilon}$. These balls are the same if we take two integers in the same block of $\Pi^{\epsilon}$, so we will only need to consider one integer $i$ representing each block of $\Pi^{\epsilon}.$

Let us check that this covers all of the proper leaves of $\T$. Let $L$ be a proper leaf and $(i(t))_{0\leq t\leq ht(L)}$ be any sequence of integers such that, for all $0\leq t \leq ht(L)$, $(i(t),t)\leq L$ in $\T$. By definition of a proper leaf, $|\Pi_{(i(t))}(t)|$ does not suddenly jump to zero, so there exists a $t<ht(L)$ such that $0<|\Pi_{(i(t))}(t)|\leq \epsilon$. This implies that $L$ is in the closed ball centered at $b_{i(t)}^{\epsilon}$ with radius $\tau_{(i(t))}^{\epsilon}.$

The covering is also \emph{fine} in the sense that $\sup_i \tau_i^{\epsilon}$ goes to $0$ as $\epsilon$ goes to $0$; indeed, if that wasn't the case, one would have a sequence $(i_n)_{n\in\N}$ and a positive number $\eta$ such that $\tau^{2^{-n}}_{i_n}\geq\eta$ for all $n$. By compactness, one could then take a limit point $x$ or a sequence $(b^{2^{-n}}_{i_n})_{n\in\N}$ , and we would have $\mu(\T_x)=0$ despite $x$ not being a leaf, a contradiction.

Now, for $0<\gamma\leq 1$, we have, summing one integer $i$ per block of $\Pi^{\epsilon}$, and using the extended fragmentation property with the stopping line $(t_i^{\epsilon})_{i\in\N}$,
\begin{align*}
E\left[\sum_{i\in \rep(\Pi^{\epsilon})} (\tau_{(i)}^{\epsilon} )^{\frac{\gamma}{|\alpha|}}\right] &\leq E\left[\sum_{i\in \rep(\Pi^{\epsilon})} E\big[\tau^{\gamma/|\alpha|}\big]|\Pi^{\epsilon}_{(i)}|^{\gamma} \right]  \\
																														 &\leq E\big[\tau^{\gamma/|\alpha|}\big] E\left[\sum_{i\in \rep(\Pi^{\epsilon})}  |\Pi^{\epsilon}_{(i)}|^{\gamma}\right].
\end{align*}
Since $\tau$ has exponential moments (see \cite{H03}, Proposition 14), the first expectation is finite and we only need to check when the second one is finite. Since $\Pi^{\epsilon}$ is an exchangeable partition, we know that, given its asymptotic frequencies, the asymptotic frequency of the block containing $1$ is a size-biased pick among them and we therefore have
\begin{align*} E\left[\sum_i |\Pi^{\epsilon}_i|^{\gamma}\right]         &= E\big[|\Pi^{\epsilon}_1|^{\gamma-1}\mathbf{1}_{\{|\Pi^{\epsilon}_1|\neq 0\}}\big] \\
																														 &= E\big[|\Pi_1(T_\epsilon)|^{\gamma-1}\mathbf{1}_{\{|\Pi_1(T_\epsilon)|\neq 0\}}\big]\\
																														 &\leq E\big[|\Pi_1(T_0^-)|^{\gamma-1}\big],
\end{align*}                                                       
where $T_{\epsilon}=\inf\{t,|\Pi_1(t)|\leq\epsilon\}$ and $T_0=\inf\{t,|\Pi_1(t)|=0\}$. Now recall that, up to a time-change which does not concern us here, the process $(|\Pi_1(t)|_{t\geq0})$ is the exponential of the opposite of a killed subordinator $(\xi(t))_{t\geq 0}$ with Laplace exponent $\phi$. This last expectation can be easily computed: let $k$ be the killing rate of $\xi$ and $\phi_0=\phi-k,$ $\phi_0$ is then the Laplace exponent of a subordinator $\xi'$ which evolves as $\xi$, but is not killed. By considering an independent random time $T$ following the exponential distribution with parameter $k$ and killing $\xi'$ at time $T$, one obtains a process with the same law as $\xi$. Thus, we have

	\[E[\e^{-(\gamma-1)\xi_{T^-}}]= E[\e^{-(\gamma-1)\xi'_{T^-}}]=\int_0^\infty k\e^{-kt} \e^{-t(\phi_0(\gamma-1))}dt=\int_0^\infty k\e^{-\phi(\gamma-1)t}dt.
\]

Thus, if $\psi(\gamma)>0$, then $\frac{\gamma}{|\alpha|}$ is greater than the Hausdorff dimension of the leaves of $\T$.
\end{proof}

\section{Some comments and applications}
\subsection{Comparison with previous results}
In \cite{HM04}, the dimension of some conservative fragmentation trees was computed. The result was, as expected, $\frac{1}{|\alpha|}$, but this was obtained with very different assumptions on the dislocation measure:

\begin{prop} Let $\nu$ be a conservative dislocation measure, $\alpha<0$, and let $\T$ be a fragmentation tree with parameters $(\alpha,0,\nu)$. Assume that $\nu$ satisfies the assumption $\mathbf{(H')}$ which we define by
	\[\int_{\s}(s_1^{-1}-1)d\nu(\mathbf{s}) < \infty.
\]
Then, almost surely, we have
	\[dim_{\mathcal{H}}(\mathcal{L}(\T))=\frac{1}{|\alpha|}.
\]
\end{prop}

This result complements ours - neither $\mathbf{(H)}$ nor $\mathbf{(H')}$ is stronger than the other, which we are going to show by producing two corresponding examples.

For all $n\geq 2$, let $s_1^n=1-\frac{1}{n}$ and, for $i\geq2$, $s_i^n=\frac{S}{n}\frac{1}{i(\log(i))^2}$, where $S=\Big(\sum_{i=2}^{\infty}\frac{1}{(i(\log(i))^{2})}\Big)^{-1}$ (this ensures that $\sum_i s_i^n=1$). Let then $\mathbf{s}^n=(s_i^n)_{i\in\N}\in\s$ and
	\[\nu_1=\sum_{n\geq 2} \frac{1}{n}\delta_{\mathbf{s}^n}.
\]
We will show that this $\sigma$-finite measure on $\s$ is a dislocation measure which satisfies $\mathbf{(H')}$ but not $\mathbf{(H)}$. First, $\int_{\s} (1-s_1)d\nu_1(\mathbf{s})=\sum_{n\geq 2}\frac{1}{n^2}<\infty$ so we do have a dislocation measure. Next, let us check $\mathbf{(H')}$: 
	\[\int_{\s}(s_1^{-1}-1)d\nu_1(\mathbf{s})=\sum_{n\geq 2}\frac{1}{n}(\frac{n}{n-1}-1)=\sum_{n\geq 2}\frac{1}{n(n-1)}<\infty.
\]
Finally, $\mathbf{(H)}$ is not verified: indeed, for any $p<1$, $n\geq 2$ and $i\geq 2$, $(s^n_i)^{p}=\frac{S^p}{n^p}\big(i(\log(i))^{2}\big)^{-p}$ which is the general term of a divergent series.

\medskip

Now we are going to do the same on the other side. For all $n\in\N$, let $t_1^n=\frac{1}{n}$ and, for $i\geq2$, let $t_i^n=T(1-\frac{1}{n})\frac{1}{i^2}$, where $T=\big(\sum_{i=2}^{\infty}(\frac{1}{i^2})\big)^{-1}$. Since $t_2^n>t_1^n$ for large $n$, the sequence $\mathbf{t}^n=(t_i^n)_{i\in\N}$ is not a mass partition (despite its sum being equal to $1$), and we will solve this problem by splitting its terms. Let $N(n)=\left\lceil\frac{t_2^n}{t_1^n}\right\rceil$, and then let $\mathbf{u}^n=(u_i^n)_{n\in\N}\in\s$ such that $u_1^n=t_1^n$ and, for $i\geq2$, $u_i^n=\frac{t_k^n}{N(n)}$ where $k$ is such that $i\in\{(k-2)N(n)+2,\ldots,(k-1)N(n)+1\}$. In other words, $\mathbf{u}^n$ starts with $t_1^n$, and then every term of $\mathbf{t}^n$ is divided by $N(n)$ and repeated $N(n)$ times. Now let us define
	\[\nu_2=\sum_{n\in\N} \frac{1}{n^2}\delta_{\mathbf{u}^n}.
\]
The measure $\nu_2$ integrates $1-s_1$ since it is finite, but $\sum_{n\in\N} \frac{1}{n^2}(\frac{1}{t_1^n}-1)=\sum_{n\in\N} \frac{1}{n}-\frac{1}{n^2}=\infty$, so $\mathbf{(H')}$ is not verified. On the other hand, for any $p<1$, we have
	\[\int_{\s} \sum_i s_i^{p^*} d\nu_2(\mathbf{s})=\sum_{n\in \N} \frac{1}{n^2} \left(\frac{1}{n^p}+N(n)\Big(\frac{T(1-\frac{1}{n})}{N(n)}\Big)^p\big(\sum_{i\geq2}\frac{1}{i^{2p}}\big)\right),
\]
which is finite as soon as $p>\frac{1}{2}$, since $N(n)$ is asymptotically equivalent to $\frac{Tn}{4}$ as $n$ goes to infinity. Thus $\nu_2$ satisfies $\mathbf{(H)}$.

\subsection{Influence of parameters on the Malthusian exponent}
We will here investigate what happens when we change some parameters of the fragmentation process. We start with a "basic" function $\psi$ to which we will add either a constant (which amounts to increasing $\nu(\{(0,0,\ldots)\})$) or a linear part (which amounts to adding some erosion). We let $p_0=\inf \{p\geq0, \psi(p)>-\infty\}$. We also exclude the trivial case where $\nu(s_2>0)=0$, where the tree is always a line segment.
\subsubsection{Influence of the killing rate}

We assume here that $\nu({(0,0,\ldots)})=0$, which implies that $\psi(0)<0$, while we do not make any assumptions on the erosion parameter $c\geq0$. We will quickly study how the Malthusian exponent changes when we add to $\nu$ a component of the form $k\delta_{(0,0,\ldots)}$ with $k\geq0$. Let therefore, for $k\geq0$, $\nu_k=\nu+k\delta_{(0,0,\ldots)}$ and, for $p\in\R$, $\psi_k(p)=cp + \int_{\s}(1-\sum_i s_i^p)d\nu_k(\mathbf{s})=\psi(p)+k$ and, if it exists, $p^*(k)$ the only number in $(0,1]$ which nulls the function $\psi_k$.

\begin{prop} Assume $\mathbf{(H)}$ for $(c,\nu)$, that is $\psi(p_0^+)<0$, and let $k_{\text{max}}=|\psi(p_0^+)|$. Then, for $k\in [0,k_{\text{max}})$, the pair $(c,\nu_k)$ also satisfies $\mathbf{(H)}$. Letting $p^*(k_{\text{max}})=p_0$ (though it is not a Malthusian exponent in our sense when $p_0=0$), the function $p^*(k)$ on $[0,k_{\text{max}}]$ is the inverse function of $-\psi$. It is thus strictly decreasing and is differentiable as many times as $\psi$.
 For $k\geq k_{\text{max}}$, $\mathbf{(H)}$ is no longer satisfied (in fact there is no Malthusian exponent if $k>k_{\text{max}}$), however we have in this case $p_0=\inf \{p\geq0, \psi_k(p)\geq 0\}$ which is the equivalent of $p'$ in Section $6.3$.
\end{prop}

\subsubsection{Influence of erosion}
Here we do not make any assumptions of $\nu$, and let, for nonnegative $c$ and any $p$, $\psi_c(p)=cp+\int_{\s}(1-\sum_i s_i^p)d\nu(\mathbf{s})$. Note that, unlike in the previous section, the standard coupling between $(\alpha,c,\nu)$-fragmentations of Section $2.1.3$ for all $c\geq0$ is such that, almost surely, if for one $c$, $\Pi^{0,c}$ dies in finite time, then $\Pi=\Pi^{0,c}$ dies in finite time for all $c$. Thus, placing ourselves on the event where they do not die in a finite time, and calling $\T_c=\TREE(\Pi^{\alpha,c})$, we have $dim_{\mathcal{H}}(\mathcal{L}(\T_c))=\frac{p^*(c)}{|\alpha|}$, $p^*(c)$ being the corresponding Malthusian exponent. 

\begin{prop} Assume $\mathbf{(H)}$ for $(0,\nu)$, that is $\psi(p_0^+)<0$. If $p_0=0$ then the couple $(c,\nu)$ satisfies $\mathbf{(H)}$ for all $c$, and its Malthusian exponent $p^*(c)$ tends to zero as $c$ tends to infinity with the following asymptotics:
	\[p^*(c)\underset{c\to\infty}\sim \frac{|\psi(0)|}{c}.
\]
If $p_0>0$, then $(c,\nu)$ satisfies $\mathbf{(H)}$ for $c<c_{\text{max}}$ with $c_{\text{max}}=\frac{|\psi(p_0^+)|}{p_0}$. By setting $p^*(c_{\text{max}})=p_0$, the function $c\to p^*(c)$ is decreasing and is differentiable as many times as $\psi$ is. For $c\geq c_{\text{max}}$ $\mathbf{(H)}$ is no longer satisfied, however we do have $p_0=\inf \{p\geq 0, \psi_k(p)\geq 0\}$.
\end{prop}

\subsection{An application to the boundary of Galton-Watson trees}

In this part we generalize some simple well-known results on the boundary of discrete Galton-Watson trees (see for example \cite{Hawkes}) to trees where the branches have exponentially distributed lengths. Unsurprisingly, the Hausdorff dimension of this boundary is the same in both cases.

Let $\xi=\sum p_i \delta_i$ be a probability measure on $\N\cup\{0\}$ which is supercritical in the sense that $m=\sum_i ip_i>1$. Let $\T$ be a Galton-Watson tree with offspring distribution $\xi$ and such that the individuals have exponential lifetimes with parameter $1$. Seeing $\T$ as an $\R$-tree, we define a new metric on it by changing the length of every edge: let $a\in(1,\infty)$ and $e$ be an edge of $\T$ connecting a parent and the child, we define the new length of $e$ to be the old length of $e$ times $a^{-n}$, where the parent is in the $n$-th generation of the Galton-Watson process. We let $d'$ be this new metric.

The metric completion of $(\T,d')$ can then be seen as $\T\cup\partial\T$ where $\partial\T$ are points at the end of the infinite rays of $\T$.

\begin{prop} On the event where $\T$ is infinite, we have
	\[\dim_{\mathcal{H}}(\partial\T)=\frac{\log m}{\log a}.
\]
\end{prop}

\begin{proof} We start with the case where there exists $N\in \N$ such that, for $i\geq N+1$, $p_i=0$. We aim to identify $(\T,d')$ as a fragmentation tree and apply Theorem \ref{00}. To do this, we first have to build a measure $\mu$ on it, as usual with Proposition \ref{01}. Let $x\in \T$, and let $n$ be its generation, we then let $m(x)=\frac{1}{N^n}$. What this means is that the mass of the whole tree is $1$, then each of the subtrees spawned by the death of the initial ancestor have mass $\frac{1}{N}$, then the death of each of these spawns trees with mass $\frac{1}{N^2}$, and so on.

We leave to the reader the details of the proof that $(\T,d',\mu)$ is a fragmentation tree, the corresponding parameters being $c=0$, $\alpha=-\frac{\log a}{\log N}$ and $\nu = \sum p_i \delta_{\mathbf{s}^i}$, with $\mathbf{s}^i=(s^i_1,s^i_2,\ldots)$ such that $s^i_j=\frac{1}{N}$ if $j\leq i$ and $s^i_j=0$ otherwise. One method of proof would be to couple $\T$ with an actual $(\alpha,0,\nu)$-fragmentation process which would be obtained by constructing the death points one by one, following the tree and choosing a branch uniformly at each branching point, which is possible since the branching points of $\T$ form a countable set.

We then just need to compute the Malthusian exponent and check condition $\mathbf{(H)}$. We are looking for a number $p^*$ such that $\int_{\s} (1 - \sum_{i=1}^{N} s_i^{p^*})d\nu(s) =0$. This can be rewritten:

\begin{align*}
\int_{\s} (1 - \sum_{j=1}^{N} s_j^{p^*})d\nu(s) &= \sum_i p_i(1 - i\frac{1}{N^{p^*}}) \\
                                                &= 1 - \frac{m}{N^{p^*}}.
\end{align*}

Thus we have $p^*=\frac{\log m}{\log N}$. Condition $\mathbf{(H)}$ is also easily checked, since $\psi(0)=1-m<0$ and we thus get
	\[\dim_{\mathcal{H}}(\partial\T) = \frac{p^*}{|\alpha|}=\frac{\log m}{\log a}.
\]

\medskip

The proof in the general case is once again done with a truncation argument, as in section $6.2.3$: once again leaving the details, we let, for all $N\in\N$, $\xi_N$ be the law of $X\wedge N$ where $X$ has law $\xi$. The monotone convergence theorem shows that the average of $\xi_N$ converges to that of $\xi$, and the tree $\T$ with offspring distribution $\xi$ can be simultaneously coupled with trees $(\T_N)_{N\in\N}$ with offspring distributions $(\xi_N)_{N\in\N}$, such that $\T$ has finite height (for its original metric) if and only if all the $(\T_N)_{N\in\N}$ also do.
\end{proof}

\begin{appendices}

\section{Proof of Proposition \ref{01}}
We will want to apply a variation of Caratheodory's extension theorem to a natural semi-ring of subsets of the tree $\T$ which generates the Borel topology. The reader is invited to look in \cite{Dud} for definitions and its Theorem 3.2.4 which is the one we will use.

\begin{defi} Let $x\in\T$, and $C$ be a finite subset of $\T_x$. We say that $C$ is a \emph{pre-cutset} of $\T_x$ if $x\leq y$ for all $y\in C$ and none of the elements of $C$ are on the same branch as another. We then let $B(x,C)=\T_x\setminus \underset{{y\in C}}\bigcup \T_{y}$. Such a set is called a \emph{pre-ball}. We let $\mathcal{B}$ be the set of all pre-balls of $\T$.
\end{defi}

Note that any set of the form $\T_x\setminus \underset{i\in[k]}\bigcup \T_{x_i}$ is a pre-ball, even if one does not specify that $\{x_i,\;i\in[k]\}$ is a pre-cutset of $x$. Indeed, if $x$ is not on the same branch as $x_i$ for some $i$, then we can remove this one from the union, if we have $x_i\leq x$ for some $i$ then we have just written the empty set, and, if for some $i\neq j$, we have $x_i\leq x_j$, we might as well remove $x_j$ from the union. All these removals leave us with a pre-cutset of $\T_x$. Also note that, given a pre-ball  $B$, there exists a unique $x\in\T$ and a unique finite pre-cutset $C$ which is unique up to reordering such that $B=B(x,C)$.

\begin{lemma} $\mathcal{B}$ is a semi-ring which contains all the $\T_x$ for $x\in\T$, and it generates the Borel $\sigma$-field of $\T$.
\end{lemma}

\begin{proof} The fact that $\mathcal{D}$ contains all the sets of the form $\T_x$ for $x\in\T$, as well as the empty set, is in the definition. Stability by intersection is easily proven: let $B\big(x,(x_i)_{i\in[k]}\big)$ and $\big(y,(y_i)_{i\in[l]}\big)$ be two pre-balls. If $x$ and $y$ are not on the same branch, then the intersection is the empty set, and otherwise, we can assume $y\geq x$, and we are left with $\T_x\setminus (\underset{i\in[k]}\bigcup \T_{x_i} \cup \underset{j\in[l]}\bigcup \T_{y_j})$ which is indeed a pre-ball.

Now let $B\big(x,(x_i)_{i\in[k]}\big)$ and $B\big(y,(y_i)_{i\in[l]}\big)$ be two pre-balls, we want to check that $\\ B\big(x,(x_i)_{i\in[k]}\big)\setminus \big(y,(y_i)_{i\in[l]}\big)$ is a finite union of disjoint pre-balls. Exceptionally, we will write here for any subset $A$ of $\T$, $\bar{A}=\T\setminus A$, for clarity's sake. We have:
\begin{align*}
B\big(x,(x_i)_{i\in[k]}\big)\cap\overline{B}\big(y,(y_i)_{i\in[l]}\big)&=\T_x\cap\underset{i\in[k]}\bigcap \bar{T}_{x_i} \cap (\bar{T_y} \cup \underset{y\in[l]}\bigcup \T_{y_i}) \\
                                                  &=(\T_x\cap \bar{\T_y}\cap\underset{i\in[k]}\bigcap \bar{T}_{x_i}) \cup \underset{y\in[l]}\bigcup (\T_x \cap \T_{y_i} \cap\underset{i\in[l]}\bigcap \bar{T}_{x_i}).
\end{align*}
Since for every $i$, $\T_x \cap \T_{y_i}$ is either equal to $\T_x$ or $\T_{y_i}$, we do have a finite union of pre-balls. This union is also disjoint, because $\bar{\T_y},\T_{y_1},\ldots,\T_{y_l}$ are all disjoint.

\smallskip
Finally, we want to check that $\mathcal{D}$ does indeed span the Borel $\sigma$-field of $\T$, which will be proven by showing that every open ball in $\T$ is the intersection of a countable amount of pre-balls. Let $x\in\T$ and $r\geq0$, and let $B$ the closed ball centered at $x$ with radius $r$. Let $y$ be the unique ancestor of $x$ such that $ht(y)=(ht(x)-r)\vee 0$. Since $\T_y$ is compact and $B\in\T_y$ is open, we know that $\T_y\setminus B$ has a countable amount of closed tree components, which we will call $(\T_{x_i})_{i\in\N}$. Writing out $B=\big(\T_y\setminus \underset{i\in\N}\cup\T_{x_i}\big) \setminus\{y\}$ then shows that it is indeed a countable intersection of pre-balls. As a consequence, there exists at most one measure on $\T$ such that $\mu(\T_x)=m(x)$ for all $x\in \T$, uniqueness in Proposition \ref{01} is proven.

\end{proof}

\begin{lemma} For every $x\in\T$ and every finite pre-cutset $C$, we let 
\[
\mu(B(x,C))=m(x)-\sum_{y\in C} m(y).
\]
This defines a nonnegative function on $\mathcal{D}$ which is $\sigma$-additive.
\end{lemma}

\begin{proof} Let us first prove the positivity of $\mu$. This can be done by induction on the number of elements $k$ in the pre-cutset $C=\{x_i,\; i\in[k]\}$ of $\T_x$. If $k=0$ then there is nothing to do, $\mu(B(x,\emptyset))=m(x)\geq0$ by definition. Now assume $k\geq 1$ and that the positivity has been proved for $k-1$. Let $y$ be the greatest common ancestor of all the $(x_i)_{i\in[k]}$, we have $x\leq y$, and thus $m(x)\geq m(y)$, and it will suffice to prove $m(y)-\sum_{i=1}^k m(x_i)\geq0$. The set $T_y\setminus\{y\}$ has a finite, but strictly greater than $1$ number of connected components which contain the points $(x_i)_{i\in[k]}$, let us call them $C_1,\ldots,C_l$, with $1\leq l\leq k.$ Since every $C_l$ contains at most $l-1\leq k-1$ elements from the $(x_i)_{i\in[k]}$, one can use the induction hypothesis in every $C_j$: for all $j$, let $y_j \in C_j$ be such that, for all $i$ such that $x_i\in C_j$, $y_j\leq x_i$, then we have $m(y_j)\geq \underset{i:\,x_i\in C_j}\sum m(x_i)$. Now, by letting every $y_j$ converge to $y$, we end up with
	\[m(y)\geq m(y^+) \geq \sum_j \lim_{y_j\to y^+} m(y_j) \geq \sum_i{m(x_i)}
\]
which ends the proof of the positivity of $\mu$.

\smallskip
The proof that $\mu$ is $\sigma$-additive on $\mathcal{D}$ will be done in three steps. First, we will prove that it is finitely additive, i.e. that, if a pre-ball can be written as a finite disjoint union of pre-balls, then the $\mu$-masses add up properly. Next, we will prove that it is finitely subadditive, which means that if a pre-ball $B$ can be written as a subset of the finite union of other pre-balls $B_1,\ldots,B_n$, we have $\mu(B)\leq \sum_i \mu(B_i)$. The $\sigma$-additivity itself will then be proved  by proving both inequalities separately. 

First, we want to show that $\mu$ is finitely additive, i.e. that if a pre-ball $B=B\big(x,(x_i)_{i\in[k]}\big)$ can be written as the disjoint union of pre-balls $B^j=B\big(x^j,(x_i^j)_{i\in[k^j]}\big)$ for $1\leq j\leq n$, we have $\mu(B)=\sum_j \mu(B^j)$. Note that since $\mathcal{D}$ is not stable under union, one cannot simply prove this for $n=2$ and then do a simple induction. We will indeed do an induction on $n$, but it will be a bit more involved. The initial case, $n=1$ is immediate. Now assume that $n\geq 2$ and that, for every pre-ball which can be written as the disjoint union of fewer than $n-1$ pre-balls, the masses add up, and let $B=B\big(x,(x_i)_{i\in[k]}\big)$ be a pre-ball which is the union of $B^j=B\big(x^j,(x_i^j)_{i\in[k^j]}\big)$ for $1\leq j\leq n$. We are first going to show that we can restrict ourselves to the case where $B=\T_x$. To do this, first notice that, since the union is disjoint, for every $i$ with $1\leq i \leq k$, there is only one $j$, which we will call $j(i)$, such that $x_i$ is in the set $\{x_p^j,\; p\in[k^j]\}$. Thus, if we add $T_{x_{i}}$ to the pre-ball $B^{j(i)}$ and do this for all $i$, the result is that $T_x$ (which is none other than $B \cup \underset{1\leq i\leq k}\cup T_{x_i}$) is written as the disjoint union of pre-balls $A^j=B^j \cup \underset{i:j(i)=j}\cup T_{x_i}$. Since $\mu(T_x)=\mu(B)+\sum_{i=1}^k m(x_i)$ and, for all $j$, $\mu(A^j)=\mu(B_j)+ \underset{i:j(i)=j}\sum m(x_i)$, it suffices consider the case when $B=\T_x$. By reordering, one can also assume that $x^1=x$. Now, for every $i$ with $1\leq i\leq k^1$, consider the pre-balls $B^j$ with $j$ such that $x^1_i\leq x_j$. These are disjoint, and their union is none other than $\T_{x^1_i}$, and they are strictly less than $n$ in number. The induction hypothesis then tells us that $\mu(T_{x^1_l})$ is the sum of $\mu(B^j)$ for such $j$. Repeat this for all $i$, and we get $\sum_{j=2}^{n} \mu(B^j)=\sum_{i=1}^{k^1} \mu(\T_{x^1_i})=\mu(\T_x)-\mu(B^1)$, which is what we wanted.

Now we go on to $\mu$'s finite subadditivity. This can actually be proven with pure measure theory. Let $B$ be a pre-ball and $B_1,\ldots,B_n$ be pre-balls such that $B\subset \underset{i\in[n]}\cup B_i$. Let us first start with the case where $n=1$, in other words, let us show that $\mu$ is nondecreasing: since $\mathcal{D}$ is a semi-ring, $B_1\setminus B$ can be rewritten as a finite disjoint of pre-balls $C_1,\ldots,C_k$, and by finite additivity, we have $\mu(B_1)=\mu(B)+\sum_j \mu(C_j) \geq \mu(B)$. Now, going back to the general case, one can assume that for every $i$, we have $B_i\subset B$, because if it is not the case, one can replace $B_i$ by $B_i\cap B$. Now, consider the sequence $C_i$ defined by $C_1=B_1$ and, for $i\geq 2$, $C_i=B_i\setminus (B_1\cup B_2\ldots\cup B_{i-1})$. Since $\mathcal{D}$ is a semi-ring, every $B_i$ can be written as the disjoint union of a finiteamount of pre-balls: for every $i$, there exists disjoint pre-balls $D_1(i),\ldots,D_{k(i)}(i)$ such that $C_i=\underset{j=1}{\overset{k(i)}\cup} D_j(i)$. By finite additivity, we then have $\mu(B)=\sum_{i=1}^{n}\sum_{j=1}^{k(i)} \mu(D_j(i))$. Now all that is left to do is show that, for all $i$, we have $\underset{j=1}{\overset{k(i)}\sum} \mu(D_j(i)) \leq \mu(B_i)$, which is immediate because $B_i\setminus (\underset{j=1}{\overset{k(i)}\cup} D_j(i))$ is a disjoint finite union of pre-balls.

Finally, we can move on to $\mu$'s $\sigma$-additivity . Assume that a pre-ball $B=B\big(x,(x_i)_{i\in[k]}\big)$ can be written as the disjoint union of pre-balls $B^j=B\big(x^j,(x_i^j)_{i\in[k^j]}\big)$ for $j\in\N$. Let us first prove the easy inequality $\mu(B)\geq \sum_i \mu(B_i)$. Fix $n\in\N$, since $\mathcal{B}$ is a semi-ring, the set $B\setminus(\underset{1\leq i\leq n} \cup B_i)$ is a finite disjoint union of pre-balls, which we will call $C_1,\ldots,C_k$. By finite additivity, we have $\mu(B)=\sum_{i=1}^n \mu(B_i)+\sum_{j=1}^{k}\mu(C_j)\geq \sum_{i=1}^n \mu(B_i)$, and we just need to take the limit. To prove the reverse inequality, we will slightly modify our sets so that we can get a open cover of a compact set, and bring ourselves back to the finite case. Let $\epsilon>0$. For every $j$ such that $x^j\neq\rho$ (and $\epsilon$ small enough), let $x^j(\epsilon)$ be an ancestor of $x^j$ such that $m(x^j(\epsilon))-m(x^j)\leq\epsilon 2^{-j-1}$, and if $x_j=\rho$ we keep $x^j(\epsilon)=\rho$. In the same vein, for $1\leq i \leq k$, we choose an ancestor $x_i(\epsilon)$ such that $m(x_i(\epsilon))-m(x_i)\leq \frac{1}{k}$, and such that $(x_i(\epsilon))_{i\in[k]}$ is still a pre-cutset of $\T_x$. Now consider, for every $j$, the open set $D^j$ which is equal to $B\big(x^j(\epsilon),(x_i^j)_{i\in[k^j]}\big)\setminus\{x^j(\epsilon)\}$ if $x^j\neq\rho$, and equal to $B^j$ otherwise. These form a cover of $B\big(x,(x_i(\epsilon))_{i\in[k]}\big)$ and therefore also cover its closure, $B\big(x,(x_i(\epsilon))_{i\in[k]}\big)\cup \underset{1\leq i\leq k}\bigcup \{x_i(\epsilon)\}$. Since $\T$ is compact, $B\big(x,(x_i(\epsilon))_{i\in[k]}\big)$ can be covered by a finite amount of the $D^j$, which we can assume are $D^1,\ldots,D^n$. We can then use finite subadditivity:

\[
 \begin{array}{ll}  \mu(B) =m(x) - \displaystyle\sum_{i=1}^k m(x_i)  &\leq m(x) - \displaystyle\sum_{i=1}^k m(x_i(\epsilon)) + \epsilon \\
       \leq \mu(B(x,(x_i(\epsilon))_{i\in[k]})) + \epsilon  &\leq \displaystyle\sum_{j=1}^n \mu(D^j) +\epsilon \\
       \leq \;\;\;\;\;\displaystyle\sum_{j=1}^\infty \mu(D^j) + \epsilon  &\leq \displaystyle\sum_{j=1}^\infty \left(\mu(B^j)+\epsilon2^{-j-1}\right) + \epsilon \\
       \leq \;\;\;\;\; \displaystyle\sum_{j=1}^\infty \mu(B^j) + 2\epsilon. 
\end{array}
\]
This gives us our final inequality.

\end{proof}

Theorem 3.2.4 of \cite{Dud} ends the proof of Proposition \ref{01}.

\section{Possibly infinite Galton-Watson processes}
The purpose of this section is to extend the most basic results from the theory of discrete time Galton-Watson processes to the case where one parent may have an infinite amount of children. We refer to \cite{Ha} for the classical results. Let $Z$ be a random variable taking values in $\mathbb{N}\cup\{0\}\cup\{\infty\}$ with $P(Z\geq1)\neq 1$, and $(Z_n^i)_{i,n\in\mathbb{N}}$ be independent copies of $Z$. Let also, for $x\geq0$, $F(x)=E[x^Z]$. We define the process $(X_n)_{n\in\mathbb{N}}$ by $X_1=1$ and, for all $n$, $X_{n+1}=\sum_{i=1}^{X_n} Z_n^i$. 

\begin{prop} The following are all true:

(i) Almost surely, $X$ either hits $0$ in finite time or tends to infinity.

(ii) If $X$ hits the infinite value once, then it stays there almost surely.

(iii) If $E[Z]>1$ then the function $F$ has two fixed points on $[0,1]$: one is the probability of extinction $q$, and the other is $1$. If $E[Z]\leq 1$ then $q=1$ and F only has one fixed point.

\end{prop}

\begin{proof} The proof of $(i)$ is the same proof as in the classical case. For $(ii)$, it is only a matter of seeing that, if we have $X_n=\infty$ for some $n$, then $P(Z=0)\neq1$ and $E[Z]>0$, thus $X_{n+1}$ is infinite by the law of large numbers. For $(iii)$, in the case where $P(Z=\infty)\neq0$, we first show that $q\neq1$ by taking an integer $k$ such that $E[\min(Z,k)]>1$, and noticing that $X$ dominates the classical Galton-Watson process where we have replaced, for all $n$ and $i$, $Z_n^i$ by $\min(Z_n^i,k)$, which is supercritical and thus has an extinction probability which is different from $1$. Then, the fact that $q$ is a fixed point of $F$ and that $F$ has at most two fixed points on $[0,1]$ are proved the same way as in the classical case.
\end{proof}

\end{appendices}
\bibliographystyle{ieeetr}

\bibliography{bibnotes}

\begin{thebibliography}{10}

\bibitem{HM04}
B.~Haas and G.~Miermont, ``The genealogy of self-similar fragmentations with
  negative index as a continuum random tree.,'' {\em Electron. J. Probab},
  vol.~9, no.~4, pp.~57--97, 2004.

\bibitem{A3}
D.~Aldous, ``The continuum random tree {III}.,'' {\em Ann. Probab}, vol.~21(1),
  pp.~248--289, 1993.

\bibitem{Kingman}
J.~F.~C. Kingman, ``The coalescent.,'' {\em Stochastic Process. Appl}, vol.~13,
  pp.~235--248, 1982.

\bibitem{Ber}
J.~Bertoin, {\em Random fragmentation and coagulation processes}, vol.~102.
\newblock Cambridge University Press, 2006.

\bibitem{B01}
J.~Bertoin, ``Homogeneous fragmentation processes,'' {\em Probab. Theory Relat.
  Field.}, vol.~121, pp.~301--318, 2001.

\bibitem{B02}
J.~Bertoin, ``Self-similar fragmentations,'' {\em Ann. Inst. Henri
  Poincar{\'e}}, vol.~38, pp.~319--340, 2002.

\bibitem{JSh}
J.~Jacod and A.~Shiryaev, {\em Limit Theorems for Stochastic Processes}.
\newblock Grundlehren der Mathematischen Wissenschaften, Springer, second~ed.,
  1987.

\bibitem{CPY97}
P.~Carmona, F.~Petit, and M.~Yor, ``On the distribution and asymptotic results
  for exponential functionals of {L}\'evy processes,,'' {\em Rev. Mat.
  Iberoamericana}, pp.~73--130, 1997.

\bibitem{Ev}
S.~Evans, {\em Probability and real trees: {\'E}cole D'{\'E}t{\'e} de
  Probabilit{\'e}s de {S}aint-{F}lour {XXXV} - 2005}, vol.~1920 of {\em Lecture
  notes in mathematics}.
\newblock Springer, 2008.

\bibitem{Bill}
P.~Billinglsey, {\em Convergence of probability measures}.
\newblock Wiley Series in Probability and Statistics, Wiley, 2009.

\bibitem{EPW06}
S.~Evans, J.~Pitman, and A.~Winter, ``Rayleigh processes, real trees, and root
  growth with re-grafting.,'' {\em Probab. Theory Related Fields}, vol.~134(1),
  pp.~918--961, 2006.

\bibitem{ADH}
R.~{Abraham}, J.-F. {Delmas}, and P.~{Hoscheit}, ``A note on
  {G}romov-{H}ausdorff-{P}rokhorov distance between (locally) compact measure
  spaces,'' {\em ArXiv e-prints}, 2012.

\bibitem{Dud}
R.~M. Dudley, {\em Real Analysis and Probability}.
\newblock Cambridge University Press, 2002.

\bibitem{HP11}
C.~Haulk and J.~Pitman, ``A representation of exchangeable hierarchies by
  sampling from random real trees..'' arXiv eprint, 2011.

\bibitem{BG}
J.~Bertoin and A.~Gnedin, ``Asymptotic laws for nonconservative self-similar
  fragmentations,'' {\em Electronic Journal of Probability}, vol.~9,
  pp.~575--593, 2004.

\bibitem{BR}
J.~Bertoin and A.~Rouault, ``Discretization methods for homogeneous
  fragmentations.,'' {\em J. London Math. Soc.}, vol.~72: no.1, pp.~91--109,
  2005.

\bibitem{DM}
C.~Dellacherie and P.~Meyer, {\em Probabilit{\'e}s et potentiel: Chapitres V
  {\`a} VIII: Th{\'e}orie des martingales}.
\newblock Actualit{\'e}s scientifiques et industrielles, Hermann, 1980.

\bibitem{Lep}
D.~L\'epingle, ``La variation d'ordre p des semimartingales.,'' {\em Z.
  Wahrscheinlichkeitstheorie verw. Gebiete}, vol.~36, pp.~295--316, 1976.

\bibitem{RY}
D.~Revuz and M.~Yor, {\em Continuous Martingales and {B}rownian Motion}.
\newblock Grundlehren der Mathematischen Wissenschaften, Springer, third~ed.,
  2004.

\bibitem{BY02}
J.~Bertoin and M.~Yor, ``On the entire moments of self-similar {M}arkov
  processes and exponential functionals of {L}\'evy processes,'' {\em Ann. Fac.
  Sci. Toulouse VI. Ser. Math.}, vol.~11: no.1, pp.~33--45, 2002.

\bibitem{BY01}
J.~Bertoin and M.~Yor, ``On subordinators, self-similar {M}arkov processes and
  some factorizations of the exponential variable,'' {\em Electron. Commun.
  Probab.}, vol.~6, pp.~no. 10, 95--106, 2001.

\bibitem{M09}
G.~Miermont, ``Tessellations of random maps of arbitrary genus.,'' {\em Ann.
  Sci. Éc. Norm. Supér.}, vol.~42, no.~5, pp.~725--781, 2009.

\bibitem{Falconer}
K.~Falconer, {\em Fractal Geometry}.
\newblock John Wiley \& Sons, 1990.

\bibitem{H03}
B.~Haas, ``Loss of mass in deterministic and random fragmentations.,'' {\em
  Stoch. Proc. App.}, vol.~106, pp.~411--438, 2003.

\bibitem{Hawkes}
J.~Hawkes, ``Trees generated by a simple branching process,'' {\em J. London
  Math. Soc.}, vol.~s2-24, no.~2, pp.~373--384, 1981.

\bibitem{Ha}
T.~E. Harris, {\em The Theory of Branching Processes}.
\newblock Springer, 1963.

\end{thebibliography}

\end{document}